\documentclass{amsart}
\usepackage{graphicx}
\title[Isometric deformation of wave fronts]{%
  Isometric deformations of wave 
  fronts at non-degenerate singular points
}
\date{April 15, 2020}

\author{A.~Honda}
\address[Atsufumi Honda]{
   Department of Applied Mathematics, 
   Faculty of Engineering, Yokohama National University,
   79-5 Tokiwadai, Hodogaya, Yokohama 240-8501, Japan
}
\email{honda-atsufumi-kp@ynu.ac.jp}

\author{K.~Naokawa}
\address[Kosuke Naokawa]{%
   Department of Computer Science, 
   Faculty of Applied Information Science,
   Hiroshima Institute of Technology,  
   2-1-1 Miyake, Saeki, Hiroshima, 731-5193, Japan
}
\email{k.naokawa.ec@cc.it-hiroshima.ac.jp}

\author{M.~Umehara}
\address[Masaaki Umehara]{%
   Department of Mathematical and Computing Sciences,
   Tokyo Institute of Technology
   2-12-1-W8-34, O-okayama, Meguro-ku,
   Tokyo 152-8552, Japan
}
\email{umehara@is.titech.ac.jp}

\author{K.~Yamada}
\address[Kotaro Yamada]{%
   Department of Mathematics,
   Tokyo Institute of Technology,
   O-okayama, Meguro, Tokyo 152-8551,
   Japan
}
\email{kotaro@math.titech.ac.jp}

\subjclass[2010]{Primary 57R45; Secondary 53A05.}
\keywords{%
   wave front, 
   isometric deformation, 
   Kossowski metric, 
   cuspidal edge, swallowtail,
   cuspidal cross cap
}

\thanks{%
  The first author was partially supported by the
  Grant-in-Aid for Young Scientists (B), No.~16K17605,
  the second author 
  was partially supported by 
  Grant-in-Aid for Young Scientists (B), No.~17K14197,
  the third author was partially 
  supported by the Grant-in-Aid for 
  Scientific Research (A) No.\ 262457005, 
  and
  the fourth author by (C) No.\ 26400087 from JSPS.
}%
\numberwithin{equation}{section}
\theoremstyle{plain}
 \newtheorem{theorem}{Theorem}
 \newtheorem{proposition}{Proposition}
 \newtheorem{lemma}{Lemma}
 \newtheorem{corollary}{Corollary}
 \newtheorem{fact}{Fact}
 \newtheorem{introtheorem}{Theorem}
 
 \newtheorem{introcorollary}[introtheorem]{Corollary}
 \newtheorem{question}{Question}
\theoremstyle{definition}
 \newtheorem{definition}{Definition}
\theoremstyle{remark}
 \newtheorem{remark}{Remark}
 \newtheorem{example}{Example}
 \newtheorem*{ack}{Acknowledgements}

\newcommand{\dy}{\displaystyle}
\newcommand{\vect}[1]{\boldsymbol{#1}}
\newcommand{\inner}[2]{\langle{#1},{#2}\rangle}
\newcommand{\pmt}[1]{{\begin{pmatrix} #1  \end{pmatrix}}}

\newcommand{\R}{\boldsymbol{R}}
\newcommand{\E}{\mathcal{E}}
\newcommand{\X}{\mathcal{X}}
\newcommand{\N}{\mathcal{N}}
\renewcommand{\P}{\mathcal{P}}
\newcommand{\D}{\mathcal{D}}

\renewcommand{\epsilon}{\varepsilon}
\renewcommand{\phi}{\varphi}

\newcommand{\sign}{\operatorname{sign}}

\begin{document}
\begin{abstract}
 Cuspidal edges and swallowtails are typical
 non-degenerate singular points  on wave fronts 
 in the Euclidean $3$-space.
 Their first fundamental forms
 belong to a class of positive semi-definite metrics
 called ``Kossowski metrics''.
 A point where a Kossowski metric
 is not positive definite is called a
 \emph{singular point} or a \emph{semi-definite point} 
 of the metric.
 Kossowski proved that real analytic 
 Kossowski metric germs at their non-parabolic singular points
 (the definition of  ``non-parabolic singular point'' is
 stated in the introduction here)
 can be realized as wave front germs
 (Kossowski's realization theorem).

 On the other hand, in a previous work with K. Saji, 
 the third and the fourth authors
 introduced the notion of
 ``coherent tangent bundle''.
 Moreover, the authors, with M. Hasegawa and K. Saji,
 proved that a Kossowski metric canonically induces 
 an associated coherent tangent bundle.

 In this paper, we shall explain Kossowski's 
 realization theorem from the viewpoint of
 coherent tangent bundles. 
 Moreover, as refinements of it, we
 give a criterion that a given Kossowski metric
 can be realized as the induced metric
 of a germ of cuspidal edge 
 (resp.\ swallowtail  or cuspidal cross cap).
 Several applications of these criteria are given.
 Also, some remaining problems on isometric deformations of
 singularities of analytic maps are given 
 at the end of this paper.
\end{abstract}
\maketitle
\section*{Introduction}
Throughout this paper, we shall treat $C^\infty$-differentiable 
objects as well as real analytic ones.
By the terminology ``$C^r$-differentiable''
we mean real analyticity if $r=\omega$ and $C^\infty$-differentiability 
if $r=\infty$.

\begin{figure}[t!]
 \begin{center}
   \begin{tabular}{c@{\hspace{2cm}}c}
        \includegraphics[width=2.5cm]{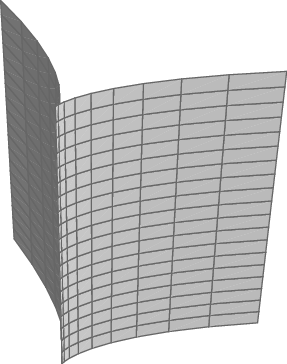} &
        \includegraphics[width=2.3cm]{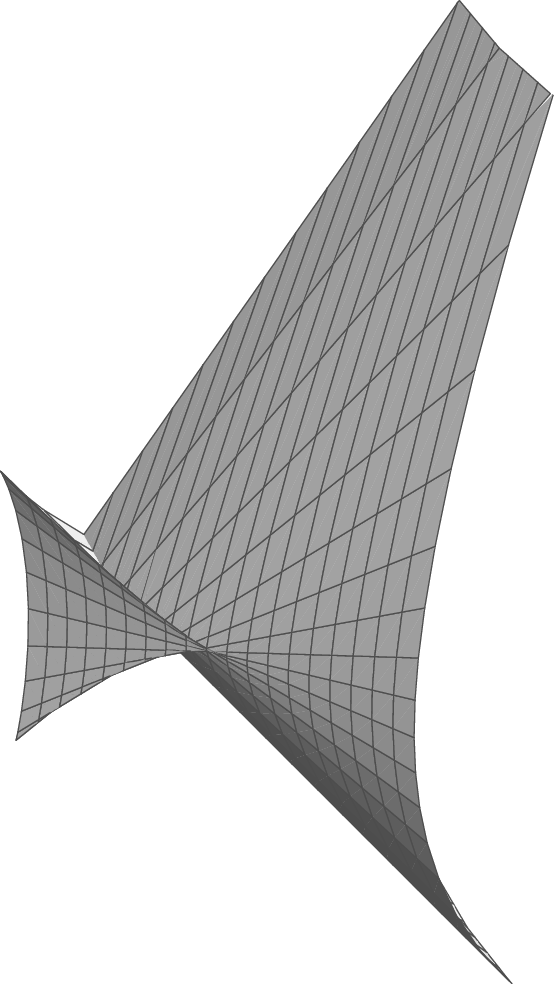} 
    \end{tabular}
  \caption{A cuspidal edge and a swallowtail.}\label{fig:0}
 \end{center}
\end{figure}

We denote by $\R^3$ the Euclidean $3$-space.
Let $M^2$ be a $C^r$-differentiable $2$-manifold and
$f:M^2\to \R^3$ a $C^r$-map.
A point $p\in M^2$ is called a \emph{singular point}
if $f$ is not an immersion at $p$.
A singular point $p\in M^2$ is called a \emph{cuspidal edge}
(resp.\ \emph{swallowtail}) if
there exist a local $C^r$-coordinate system
$(u,v)$ centered at $p$ 
and a local $C^r$-diffeomorphism $\Phi$ on $\R^3$
such that (cf.\ Figure \ref{fig:0})
\begin{align}
 \label{eq:c}
 \Phi\circ f(u,v)&=(u^2,u^3,v)\,(=:f_{\rm \tiny C}),\\
 \label{eq:sw}
 (\mbox{resp.}\,\, \Phi\circ f(u,v)&=(3u^4+u^2v, 4u^3+2uv, v)
 (=:f_{\rm \tiny SW})).
\end{align}
A $C^r$-map $f:M^2\to \R^3$ is called 
a (co-orientable) \emph{frontal} 
if there exists a $C^r$-differentiable unit vector 
field $\nu$ along $f$ such that $\nu(p)\in \R^3$ is
perpendicular to $df(T_pM^2)$ for each $p\in M^2$,
where $T_pM^2$ is the tangent space of $M^2$ at $p$.
Such a $\nu$ is called a {\it unit normal vector
field} along $f$, and can be identified with 
the \emph{Gauss map}
\[
   \nu:M^2\to S^2
\]
by parallel transport in $\R^3$,
where
\begin{equation}\label{eq:s2}
 S^2:=\{(x,y,z)\in \R^3\,;\, x^2+y^2+z^2=1\}.
\end{equation}
(The unit normal vector field $\nu$ can be
 chosen up to $\pm$-ambiguity at each
 local coordinate neighborhood, in general.
 The co-orientability of $f$ is the property that
 its unit normal vector field
 can be extended as a $C^r$-differentiable
 vector field along $f$. In this paper,
 we assume that frontals  are all co-orientable.)
A ($C^r$-differentiable) frontal $f$ is called a \emph{wave front}
if the induced map defined by
\[
  L:=(f,\nu):
     M^2\ni p \mapsto (f(p),\nu(p))\in \R^3\times S^2
\]
is an immersion. 
It is well-known that cuspidal edges and swallowtails are 
typical singularities appearing
on wave fronts. A singular point $p\in M^2$ of 
a $C^r$-map $f:M^2\to \R^3$
is called a \emph{cross cap}
(resp.\ a \emph{cuspidal cross cap})
if there exist a local $C^r$-coordinate system
$(u,v)$ and a local $C^r$-diffeomorphism $\Phi$ on $\R^3$
such that (cf.\ Figure~\ref{fig:1})
\begin{align}\label{eq:w}
 &\Phi\circ f(u,v)=(u,uv,v^2)(=:f_{\rm \tiny CR}), \\
 &(\mbox{resp.}\,\,\, \Phi\circ f(u,v)
 =(u,v^2,uv^3)(=:f_{\rm \tiny CCR})).
 \label{eq:crr}
\end{align}

\begin{figure}[bht]
 \begin{center}
   \begin{tabular}{c@{\hspace{1cm}}c}
        \includegraphics[width=4.2cm]{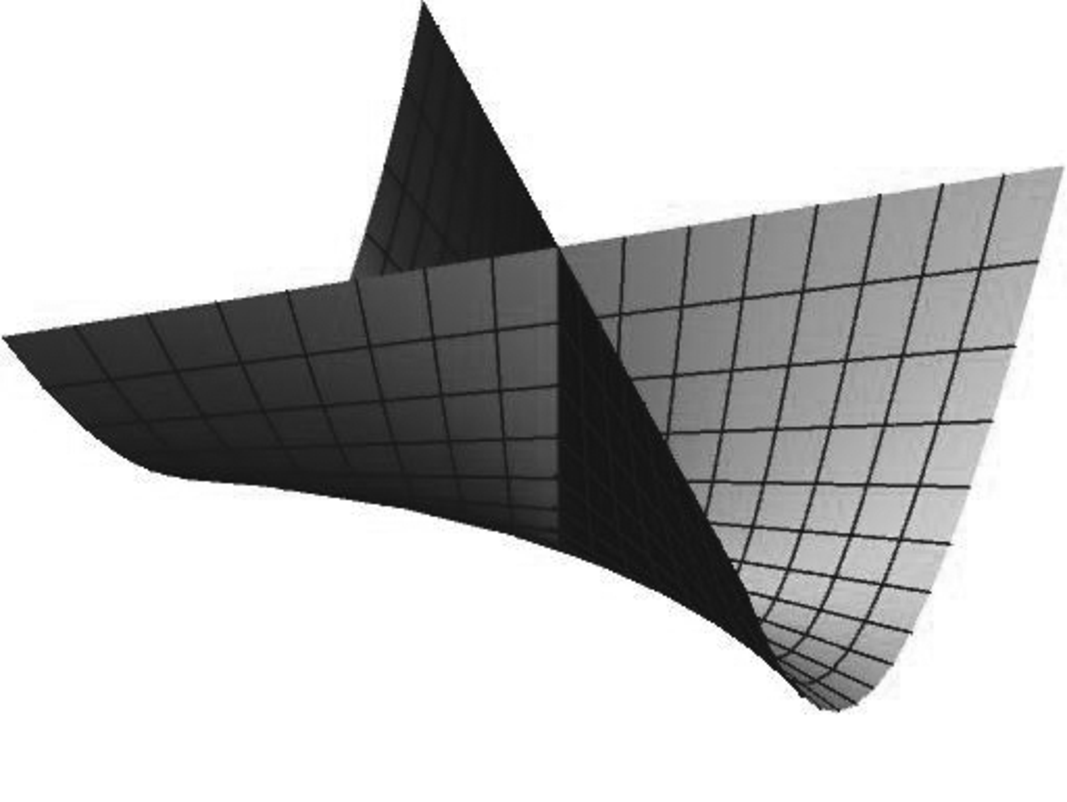} &
        \includegraphics[width=4.2cm]{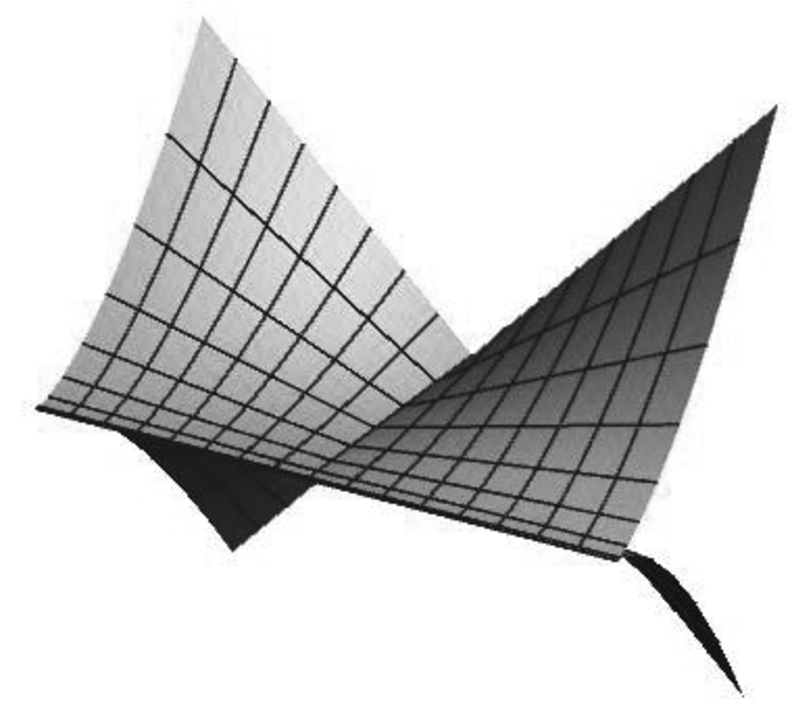}
    \end{tabular}
  \caption{A cross cap and a cuspidal cross cap.}\label{fig:1}
 \end{center}
\end{figure}

Cross caps are not frontals, since their 
unit normal vector fields cannot be extended
continuously across the singular points.
On the other hand, cuspidal cross caps are frontals, but 
not fronts.

\medskip

Let $f:M^2\to \R^3$ be a $C^r$-frontal with 
$C^r$-differentiable unit normal vector field $\nu$. 
If we take a $C^r$-differentiable local coordinate 
system $(U;u,v)$ on $M^2$, then the function
\begin{equation}\label{eq:lambda}
 \lambda:=\det(f_u,f_v,\nu)
  \qquad(f_u:= \partial f / \partial u,~
  f_v:= \partial f / \partial v)
\end{equation}
plays the role of an identifier of the singular points of $f$,
that is, $\lambda(p)=0$ if
and only if $p$ is a singular point.
We call $\lambda$ the \emph{signed area density function}
on $U$. 
A singular point $p\in U$ 
(i.e.\ the point satisfying $\lambda(p)=0$) 
is said to be \emph{non-degenerate} 
if the gradient vector $\nabla \lambda(p):=(\lambda_u(p),\lambda_v(p))$ 
does not vanish. 
If $p$ is a non-degenerate singular point,
then, by the implicit function theorem, there exists
a $C^r$-regular curve $\sigma(t)$ ($|t|<\epsilon$) on $U$ 
parametrizing the singular set of $f$ 
such that $\sigma(0)=p$.
We call the curve $\sigma$ the \emph{characteristic curve}
(or the \emph{singular curve}) passing through $p$.
Cuspidal edges, swallowtails and cuspidal cross caps
are non-degenerate singular points.

\begin{definition}\label{def:admissible}
 Let $p$ be a non-degenerate singular point of 
 a $C^r$-frontal 
 $f\colon{}M^2\to \R^3$.
 A $C^r$-differentiable 
 local coordinate system $(U;u,v)$ centered at $p$ is
 called \emph{adjusted} if $f_v(p)=\vect{0}$.
\end{definition}

We denote by ``$\cdot$'' the canonical inner product on $\R^3$,
and set $|\vect{a}|:=\sqrt{{\vect{a}}\cdot{\vect{a}}}$
($\vect{a}\in \R^3$).
Taking an adjusted coordinate system at a
non-degenerate singular point $p$,
we define
\begin{equation}\label{eq:kn0}
 \kappa_\nu(p):=\frac{{f_{uu}(p)}\cdot{\nu(p)}}{|f_u(p)|^2},
\end{equation}
which is called the \emph{limiting normal curvature}.
The definition of $\kappa_\nu(p)$
does not depend on the choice of an adjusted
coordinate system (cf.\ \cite[(2.2)]{MSUY}).  

Let $\gamma(t)$ be a curve on $M^2$
defined on an interval $I$ such that	
$\hat \gamma:=f\circ \gamma: I\to \R^3$
is a $C^r$-regular curve. 
Then the \emph{normal curvature function} 
along $\hat \gamma$ is defined by
\begin{equation}\label{eq:kn1}
 \kappa_n(t) :=\frac{{\hat \gamma''(t)}\cdot {\hat\nu(t)}}
  {|\hat \gamma'(t)|^2} \qquad (\hat \nu:=\nu \circ \gamma),
\end{equation}
where the prime ${}'$ means $d/dt$.
We let $p$ be a non-degenerate singular point, 
and let $\sigma(t)$ ($|t|<\epsilon$) be
the characteristic curve
passing through $p$ such that $p=\sigma(0)$.
As shown in \cite{MSUY}, the following assertion holds:

\begin{fact}\label{fact:k_n=k_nu}
 If $p$ is a cuspidal edge or a cuspidal cross cap
 {\rm(}resp.~a swallowtail{\rm)} 
 on a $C^r$-differentiable frontal $f$, 
 then $\hat \sigma(t):=f\circ \sigma(t)$ for 
 $t\in I$ $($resp. for $t\in I\setminus \{0\})$  
 is a $C^r$-regular curve, and the value  $\kappa_\nu(p)$ 
 coincides with the normal curvature 
 $\kappa_n(0)$ $($resp. the limit of the normal curvature
 $\dy\lim_{t\to 0}\kappa_n(t))$.
\end{fact}

\begin{definition}\label{def:gen}
 A non-degenerate singular point $p$ of a $C^r$-differentiable
 frontal 
 $f$ is said to be \emph{$\nu$-flat} 
 if its limiting normal curvature $\kappa_\nu(p)$ 
 vanishes, and is said to be \emph{non-$\nu$-flat}
 otherwise.
\end{definition}

Kossowski defined a class of positive semi-definite metrics
on $2$-manifolds.
We call metrics belonging to this class
``Kossowski metrics'' (see Definition~\ref{def:K2}).
A point where a Kossowski metric is not positive definite
is called a \emph{singular point} or a \emph{semi-definite point}
of the metric.
A Riemannian metric (i.e.\ a positive definite metric)
is a Kossowski metric without singular points.
(The concept of Kossowski
 metric can be generalized to manifolds of arbitrary dimension,
 see \cite{SUY_msj}.)

In this paper, we consider singular points of metrics
as well as singular points of $C^r$-differentiable maps.
To distinguish between these two kinds of singular points, we
use the terminology ``semi-definite points''
for singular points of a metric.  
On the other hand, a point where the metric is positive definite 
is called a \emph{regular point}.
A Kossowski metric on $M^2$
induces a $C^r$-function called a
\emph{signed area density function} (cf.\ \eqref{eq:k1}),
which is defined on each coordinate neighborhood. 
The following fact explains 
how Kossowski metrics are related to frontals
(see \cite{Kossowski} and also \cite{HHNSUY}):

\begin{fact}\label{fact:p-back}
 The first fundamental form 
 {\rm(}i.e.\ the pull-back of the canonical metric on $\R^3${\rm)}
 of a $C^r$-differentiable frontal which 
 admits only non-degenerate singular points 
 is a $C^r$-differentiable Kossowski metric. 
 Moreover, the signed area density function
 given in \eqref{eq:lambda}
 coincides with that of the  Kossowski metric
 up to $\pm$-multiple ambiguity. 
\end{fact}

Since this fact plays an important role,
we shall prove this fact in Section~\ref{sec1}.
For each semi-definite point $p$ of 
a Kossowski metric, an invariant (cf.\ \eqref{eq:euler0})
\[
   \Omega(p)\in T_p^* M^2\wedge T_p^* M^2
\]
is defined. 
If $\Omega(p)=0$,
we call $p$ a \emph{parabolic point} of $ds^2$
(cf.\ Definition \ref{def:generic}).
The following fact explains the relationship
between singular points on wave fronts
and semi-definite points on Kossowski metrics.

\begin{fact}[\cite{MSUY}]\label{fact:front}
 Let $p$ be a non-degenerate singular point of
 a $C^r$-differentiable frontal $f:M^2\to \R^3$.
 Then the following three assertions are equivalent:
 \begin{enumerate}
  \item $p$ is a non-parabolic semi-definite point of
	the induced Kossowski metric, 
  \item $f$ is a wave front at $p$, and 
	$p$ is a non-$\nu$-flat singular point of $f$,
  \item $p$ is a regular point of the Gauss map of $f$.
 \end{enumerate}
\end{fact}

Kossowski proved the following:

\begin{fact}[Kossowski's realization theorem \cite{Kossowski}]
\label{fact:K}
 Let $ds^2$ be a real analytic {\rm(}i.e.\ $C^\omega$-differentiable{\rm)} 
 Kossowski metric
 on a real analytic $2$-manifold $M^2$, and let
 $p\in M^2$ be a  non-parabolic
 semi-definite point of $ds^2$.
 Then there exist a neighborhood
 $U$ of $p$ and a real analytic wave front
 $f:U\to \R^3$ such that the first fundamental form 
 of $f$ coincides with $ds^2$ on $U$.
\end{fact}

In a joint work with Saji \cite{SUY_annals}, 
the third and the fourth authors
introduced the notion of
``coherent tangent bundle''
and proved Gauss-Bonnet type formulas for it.
A realization of the $C^r$-differentiable
vector bundle as a limiting tangent bundle of
a $C^r$-differentiable frontal is given in \cite{SUY_kodai}.
The purpose of this paper is to explain 
Kossowski's realization theorem (Fact \ref{fact:K})
from the viewpoint of the theory of coherent 
tangent bundles, and to prove several refinements. 
In fact, we define $A_2$  points and $A_3$  points 
as semi-definite points of a Kossowski metric $ds^2$
(see Definition \ref{def:a2a3add}).
The following fact is important:

\begin{fact}[{\cite[Proposition 2.19]{HHNSUY}}]
\label{fact:a2a3}
 Let $f:M^2\to \R^3$ be a $C^r$-differentiable
 wave front, and let
 $p\in M^2$ be a non-degenerate singular point.
 Then $p$ is a cuspidal edge $($resp. a swallowtail$)$
 if and only if it is an $A_2$ semi-definite point
 {\rm(}resp.\ an $A_3$ semi-definite point{\rm)} of $ds^2$.
\end{fact}

Cross caps 
are generic singular points appearing on
$C^\infty$-differentiable maps of $2$-manifolds 
into $\R^3$. 
However, they never appear on frontals 
({\cite[Proposition~4.3]{HHNSUY}}).
The corresponding assertion for cross cap singular points
is an open problem (see Question~\ref{q:3} in Section~\ref{sec6}).

If $p$ is an $A_2$ semi-definite point,
then the secondary invariant
\[
    \Omega'(p)\in T_p^* M^2\wedge T_p^* M^2
\]
is also defined (cf.\ \eqref{eq:DO}).
The following assertion holds:
\begin{introtheorem}\label{thm:main0}
 Let $M^2$ be a real analytic  $2$-manifold
 and $ds^2$ a real analytic Kossowski metric on it.
 Suppose that $p\in M^2$ is a semi-definite 
 point of the metric $ds^2$. Then there exists a 
 real analytic frontal
 $f:U\to \R^3$ defined on a neighborhood $U$ of $p$
 such that $ds^2$ is the first fundamental 
 form of $f$, and the limiting normal curvature 
 of $f$ at $p$ does not vanish.
 Moreover, such a realization $f$ satisfies  
 the following properties{\rm:}
\begin{enumerate}
 \item\label{item:main0:1}
       $f$ is a  wave front at $p$
       if and only if  $p$ is a non-parabolic
       point {\rm(}of $ds^2${\rm)},
 \item $f$ has a cuspidal edge at $p$ if and only if
       $p$ is a non-parabolic $A_2$ semi-definite point,
 \item $f$ has a swallowtail at $p$ if and only if
       $p$ is a non-parabolic $A_3$ semi-definite point,
 \item $f$ has a cuspidal cross cap at $p$ if and only if
       $p$ is a parabolic $A_2$ semi-definite point satisfying 
       $\Omega'(p)\ne 0$.
\end{enumerate}
\end{introtheorem}
Fact \ref{fact:K} corresponds to the assertion \ref{item:main0:1}.
In particular, 
Theorem \ref{thm:main0} is a generalization and refinement
of Fact \ref{fact:K}.
We prove this in Section \ref{sec4}.

\begin{remark}
 Under the assumptions of Theorem \ref{thm:main0},
 it is shown in \cite{HS} that $f$ has
 a $5/2$-cuspidal edge at $p$ if 
 the Gaussian curvature function $K$ of $ds^2$
 can be extended as a smooth function defined on
 a sufficiently small neighborhood of $p$
 and $dK(\eta)$ does not vanish at $p$,
 where $\eta\in T_pM^2$ is a null direction
 at the semi-definite point $p$.
\end{remark}

\begin{definition}\label{def:isom}
 Let $f_i$ ($i=1,2$) be two germs of $C^r$-frontals.
 Then we say these two map germs are 
 \emph{congruent} (resp.\ \emph{isometric})
 if there exist an isometry germ $\Phi$ on $\R^3$
 and a 
 $C^r$-diffeomorphism germ
 $\phi$
 (resp.\ a $C^r$-diffeomorphism germ $\phi$)  such that  
 $\Phi\circ f_2\circ \phi=f_1$
 (resp.\ $\phi^*ds^2_2=ds^2_1$),
 where $ds^2_i$ ($i=1,2$) is the 
 first fundamental form of $f_i$.
 On the other hand,
 two map germs are 
 \emph{strongly congruent} 
 if there exists an isometry germ $\Phi$ on $\R^3$
 such that $\Phi\circ f_2=f_1$.
\end{definition}

The strong congruence implies the congruence.
In this paper, we mainly discuss
the number of strong congruence classes
of wave fronts with the same first fundamental forms.
The following theorem gives properties of
the set of germs of real analytic frontals 
whose first fundamental forms coincide with 
a real analytic Kossowski metric germ $ds^2$ 
at $p\in M^2$.

\begin{introtheorem}\label{thm:main}
 Let $M^2$ be a real analytic $2$-manifold
 and $ds^2$ a real analytic Kossowski metric on $M^2$.
 Let $\omega(t)$ and $\mu(t)$ 
 be two germs of real analytic functions of one variable 
 at $t=0$.
 For each $p\in M^2$, 
 take a $C^\omega$-regular curve 
 $\gamma(t)$  
 in $M^2$
 such that $\gamma(0)=p$ and $\gamma'(0)$
 is not a null vector
 {\rm(}i.e.\
 $ds^2(\gamma'(0),\gamma'(0))> 0$, see Definition \ref{def:a2a3add}{\rm)}. 
 Then there exists a real analytic frontal germ $f=f_{\omega,\mu}$ 
 satisfying the following properties{\rm:}
 \begin{enumerate}
  \item $ds^2$ is the first fundamental form of $f$,
  \item the normal curvature function germ along 
	$\gamma$ defined by
	\eqref{eq:kn1} coincides with $e^{\omega(t)}$ 
	for a suitable
	choice of unit normal vector field $\nu$, 
  \item $\mu(t)$ gives the torsion function 
	germ along $\hat \gamma(t)=f\circ \gamma(t)$,
  \item if $p$ is a regular point $($resp.~a
	non-parabolic semi-definite point$)$ of $ds^2$,
	then $f$ is an immersion $($resp. a wave front
	with non-vanishing limiting normal curvature$)$.
 \end{enumerate}
 The possibilities for the strong congruence classes of such an 
 $f$ are at most two.
 In particular, if $\mu$ vanishes identically
 $($i.e. $\hat \gamma$ is a planar curve$)$,
 then the strong congruence class of 
 $f$ is uniquely determined.
\end{introtheorem}

\begin{remark}
 When $\gamma(t)$ is a characteristic curve of $ds^2$
 consisting of semi-definite points of type $A_2$,
 the assertion of Theorem B is proved in \cite{NUY}.
 So Theorem B can be considered as its generalization.
\end{remark}

\begin{remark}
 In the last statement of Theorem \ref{thm:main},
 we wrote that ``the possibilities for the strong 
 congruence classes of $f$ are at most two''.
 However, if we consider the
 possibilities for the congruence classes instead,
 the number turns to be ``four'' since
 we have the freedom to reverse the orientation
 of the singular curve, see \cite{Brazil}
 for details.
\end{remark}

As a consequence, the following assertion holds:

\begin{introcorollary}\label{cor:main}
 Let $I$ be an interval, and
 let $\{\omega_s(t)\}_{s\in I}$ 
 and $\{\mu_s(t)\}_{s\in I}$
 be two families of real analytic function germs 
 of the variable $t$  depending real analytically 
 on the parameter $s$.
 Then there exists a family
 $f_s:=f_{\omega_s,\mu_s}$ $(s\in I)$ 
 of real analytic frontal germs
 satisfying the properties $(1)$--$(4)$
 in Theorem B for 
 each $s\in I$ and
 depending on the parameter $s$
 real analytically.
\end{introcorollary}

In Section \ref{sec4}, we prove 
Theorems \ref{thm:main0} and  \ref{thm:main} and Corollary \ref{cor:main},
and 
also give a variant (cf.\ Theorem \ref{thml:main2})
of Theorem \ref{thm:main}.
When $p$ is an $A_2$ semi-definite point,
we can choose $\gamma$ to be a characteristic  curve,
since $\gamma'(0)$ is not a null vector.
Then we obtain the following assertion:

\begin{introcorollary}\label{cor:germ-A2}
 Let $f:(U,p)\to \R^3$ be a real analytic germ of 
 cuspidal edge {\rm(}resp.\ cuspidal cross cap{\rm)},
 and let $\sigma(t)$ be a real analytic germ of
 regular curve in $U$
 parametrizing the singular set 
 by the arc-length parameter such that $\sigma(0)=p$.
 Suppose that the limiting normal curvature at $p$
 does not vanish.
 We let $\Gamma(t)$ be a 
 real analytic
 germ of regular 
 space curve parametrized by the arc-length 
 such that the curvature function $\kappa(t)$ of $\Gamma(t)$
 is the same as that of $f\circ \sigma(t)$
 $(\Gamma(t)$ may not have the same
 torsion function as $\hat \sigma(t):=f\circ \sigma(t))$.
 Then, for each choice of $\Gamma$, there exist a 
 neighborhood $V(\subset U)$ of $p$ and a front
 $($resp. a frontal$)$ $g:(V,p)\to \R^3$ having a 
 cuspidal edge 
 {\rm(}resp.\ cuspidal cross cap{\rm)} at $p$ such that
 $g$ is isometric to $f$ and
 $\Gamma(t)=g\circ \sigma(t)$. 
 Moreover,   
 the possibilities for the strong congruence classes 
 of such a $g$ are at most two.
\end{introcorollary}

We prove Corollary \ref{cor:germ-A2} also in Section \ref{sec4}.
Here, we remark that, in \cite{HS},
analogues of Theorems \ref{thm:main0} and \ref{thm:main}
and Corollary \ref{cor:germ-A2} 
are obtained for $5/2$-cuspidal edges.
As a consequence of Theorem \ref{thm:main0}
and Theorem \ref{thm:main},
the following assertion is obtained:

\begin{introcorollary}\label{cor:cw}
 Let $f_0$, $f_1$ be two real analytic frontal  
 germs with singularities
 whose limiting normal curvatures do not vanish.
 Suppose that they are mutually isometric.
 Then there exists a continuous deformation
 of real analytic frontal germs $g_s$ $(0\le s \le 1)$
 satisfying the following properties{\rm:}
 \begin{enumerate}
  \item $g_0=f_0$ and $g_1=f_1$,
  \item $g_s$ is isometric to $g_0$,
  \item the limiting normal curvature of each $g_s$
	does not vanish.
 \end{enumerate}
 Moreover, if both $f_0$ and $f_1$
 are germs of cuspidal edges, swallowtails or  cuspidal cross caps,
 then so are $g_s$ for $0\le s\le 1$.
\end{introcorollary}

In particular, if $T$ is an orientation reversing isometry of $\R^3$,
then $T\circ f_0$ can be isometrically deformed into $f_0$
(see Remark~\ref{rmk:Tf} for details).

The paper is organized as follows:
In Section \ref{sec1}, we recall the definition of
Kossowski metrics, and define $A_2$ semi-definite points and
$A_3$ semi-definite points. The relationship between 
frontals and the induced Kossowski metrics is 
also discussed there.
In Section \ref{sec2}, we show the existence of 
certain orthogonal local coordinate systems
(called ``K-orthogonal coordinates'') for Kossowski metrics.
Using this, we show representation formulas for $A_2$ or
$A_3$ semi-definite points of Kossowski metrics.
As an application, we also discuss properties 
of distance functions induced by Kossowski metrics.
In Section \ref{sec3}, we explain the relationships
between Kossowski metrics and their induced
coherent tangent bundles.
In Section \ref{sec4}, we prove 
the main results,
using K-orthogonal coordinates. 
In Section \ref{sec6}, we mention some open questions
relating to our results.

\section{Kossowski metrics}
\label{sec1}

Throughout this paper, we fix a $C^r$-differentiable 
$2$-manifold $M^2$, where $r=\infty$ or $\omega$.
Let $ds^2$ be a positive semi-definite $C^r$-metric
on $M^2$.

\begin{definition}\label{def:sing-metric}
 A point $p\in M^2$ is called a \emph{regular point}
 of $ds^2$ if $ds^2$ is positive definite at $p$,
 and is called a \emph{singular point} 
 or \emph{semi-definite point} 
 if it is not regular.
\end{definition}

To distinguish from singular points of frontal maps,
we use the terminology \emph{semi-definite points}
for singular points of semi-definite metrics.
The set of semi-definite points in $M^2$
is called the \emph{semi-definite set}.

For the sake of simplicity, we use the notations
\begin{equation}\label{eq:ini}
 \partial_u:=\partial/\partial u, \qquad \partial_v:=\partial/\partial v 
\end{equation}
for each local coordinate system $(u,v)$ of $M^2$.

\begin{definition}\label{def:null}
 Let $p$ be a semi-definite point of the metric $ds^2$ on $M^2$.
 Then a non-zero tangent vector $\vect{v}\in T_pM^2$
 is called a \emph{null vector} if
 \begin{equation}\label{eq:vv}
  ds^2(\vect{v},\vect{v})=0.
 \end{equation}
 Moreover, a local coordinate neighborhood $(U;u,v)$ 
 is called \emph{adjusted} at $p\in U$ if
 $\partial_v$ gives a null vector
 of $ds^2$ at $p$.
\end{definition}
  
It can be easily checked that
\eqref{eq:vv} implies that $ds^2(\vect{v},\vect{x})=0$
holds for all $\vect{x}\in T_pM^2$.
If $(U;u,v)$ is a local coordinate neighborhood 
adjusted at a semi-definite point $p=(0,0)$, then 
$F(0,0)=G(0,0)=0$ holds, where
\begin{equation}\label{eq:I}
  ds^2=E\,du^2+2F\,du\,dv+G\,dv^2.
\end{equation}

We denote by $\X^r$ the set of $C^r$-differentiable 
vector fields on $M^2$, and by $C^{r}(M^2)$ 
the set of real valued $C^r$-differentiable functions on $M^2$.
We set
$\inner{X}{Y}:=ds^2(X,Y)$ for $X,Y\in \X^r$.
Kossowski \cite{Kossowski} defined a map
$\Theta:\X^r \times \X^r
\times \X^r \to C^{r}(M^2)$
as 
\begin{multline}\label{eq:Gamma}
\Theta(X,Y,Z):=\frac12
 \biggl(
 X\inner{Y}{Z}+Y\inner{X}{Z}-Z\inner{X}{Y}\\
 +\inner{[X,Y]}{Z}-\inner{[X,Z]}{Y}
 -\inner{[Y,Z]}{X}
 \biggr),
\end{multline}
where $[X,Y]:=XY-YX$.
We call $\Theta$ the \emph{Kossowski pseudo-connection}
with respect to the Kossowski metric.

If the metric $ds^2$ is positive definite, then  
$\Theta(X,Y,Z)=\langle \nabla_XY,Z\rangle$
holds, where
$\nabla$ is the Levi-Civita connection of $ds^2$.
One can easily check 
the following two identities (cf.\ \cite{Kossowski})
\begin{align}\label{eq:1}
&X\langle Y,Z\rangle=\Theta(X,Y,Z)+\Theta(X,Z,Y),\\
\label{eq:2}
&\Theta(X,Y,Z)-\Theta(Y,X,Z)=\langle [X,Y],Z\rangle.
\end{align}
The equation \eqref{eq:1} (resp.~\eqref{eq:2})
corresponds to the condition 
that  $\nabla$ is a metric connection (resp.\ is torsion free).
The following assertion can be also
easily verified:
\begin{proposition}[Kossowski \cite{Kossowski}]\label{prop:K-conn}
 For each $Y \in \X^r$
 and for each semi-definite point $p\in M^2$,
 the map
 \[
     T_pM^2\times T_pM^2\ni (\vect{v}_1,\vect{v}_2)\longmapsto
           \Theta(V_1,Y,V_2)(p)\in \R
 \]
 is a well-defined bilinear map, 
 where $V_j$ $(j=1,2)$
 are $C^r$-differentiable vector fields of $M^2$ 
 satisfying $\vect{v}_j=V_j(p)$.
\end{proposition}

For each $p\in M^2$, the subspace
\[
 \N_p:=\biggl\{\vect{v} \in T_pM^2\,;\, ds^2(\vect{v},\vect{w})=0
   \mbox{ for all $\vect{w}\in T_pM^2$}
   \biggr\}
\]
is called the \emph{null space} or the \emph{radical} of $ds^2$
at $p$.
A non-zero vector belonging to $\N_p$ is 
a null vector at $p$ (cf.\ Definition \ref{def:null}).
\begin{lemma}[Kossowski \cite{Kossowski}]\label{lem:K-conn}
 Let $p$ be a semi-definite point of $ds^2$.
 Then the Kossowski pseudo-connection $\Theta$
 induces a tri-linear map
 \[
    \hat \Theta_p:T_pM^2\times T_pM^2\times \N_p
            \ni (\vect{v}_1,\vect{v}_2,\vect{v}_3) \longmapsto
          \Theta(V_1,V_2,V_3)(p)\in \R,
 \]
 where
 $V_j$ $(j=1,2,3)$ are $C^r$-vector fields of $M^2$
 such that $\vect{v}_j=V_j(p)$.
\end{lemma}
\begin{proof}
 Applying \eqref{eq:Gamma}, 
 \begin{align*}
  & 2\Theta(V_1,fV_2,V_3) \\
  &\phantom{aaa}=
  V_1\langle fV_2,V_3\rangle+
  fV_2\langle V_1,V_3\rangle-V_3\langle V_1,fV_2
  \rangle\\
  &\phantom{aaaaaaaaa}+
  \langle [V_1,fV_2],V_3\rangle-\langle [V_1,V_3],fV_2
  \rangle
  -\langle [fV_2,V_3],V_1\rangle \\
  &\phantom{aaa}=
  2f\Theta(V_1,V_2,V_3)
  +(V_1f) \langle V_2,V_3\rangle-
  (V_3f)\langle V_1,V_2\rangle \\
  &\phantom{aaaaaaaaaaaaaaaaaaaaaaaaa}+
  (V_1f)
  \langle V_2,V_3\rangle +(V_3f) \langle V_2,V_1\rangle\\
  &\phantom{aaa}=
  2f\Theta(V_1,V_2,V_3)
    +2(V_1f) \langle V_2,V_3\rangle= 2f\Theta(V_1,V_2,V_3)
 \end{align*}
 holds at $p$,
 where the fact that $V_3(p)\in \N_p$ 
 is used to show the last equality.
\end{proof}
\begin{definition}\label{def:adms}
 A semi-definite point $p$ of the metric $ds^2$
 is called 
 \emph{admissible}\footnote{%
   Admissibility was originally introduced by 
   Kossowski \cite{Kossowski}. He called it
   $d(\langle,\rangle)$-\emph{flatness}.
} 
 if $\hat \Theta_p$ in Lemma~\ref{lem:K-conn} vanishes.
\end{definition}
A semi-definite point of the metric $ds^2$
is called of \emph{rank one} if $\N_p$
is a $1$-dimensional subspace of $T_pM^2$.

By a suitable affine transformation
in the $uv$-plane,
one can take a local coordinate system
adjusted at $p$
(cf.\ Definition \ref{def:sing-metric}).
The following assertion gives a characterization
of admissible semi-definite points:
\begin{proposition}[\cite{HHNSUY}]\label{prop:property}
 Let $(u,v)$ be a $C^r$-differentiable local coordinate system
 adjusted at a rank one semi-definite 
 point $p$.
 Then  $p$ is admissible if and only if 
 \begin{equation}\label{eq:admissible}
   F=G=0,\quad E_v = 2 F_u, \quad G_u=G_v=0 
 \end{equation}
 hold at $p=(0,0)$, 
 where $ds^2=E\,du^2+2F\,du\,dv+G\,dv^2$. 
\end{proposition}
\begin{proof}
 Since $[\partial_u,\partial_v]$ vanishes,
 and $\partial_v\in \N_p$ at $p$,
 the formula \eqref{eq:Gamma} yields that
 \begin{align*}
  &2\hat\Theta(\partial_u,\partial_u,\partial_v)
  =2\partial_u\langle \partial_u,\partial_v \rangle-
  \partial_v\langle \partial_u,\partial_u \rangle=2F_u-E_v,\\
  &2\hat\Theta(\partial_u,\partial_v,\partial_v)
  =\partial_u\langle \partial_v,\partial_v \rangle+
  \partial_v\langle \partial_u,\partial_v \rangle
  - \partial_v\langle \partial_u,\partial_v \rangle=
  \partial_u\langle \partial_v,\partial_v \rangle
  =G_u,\\
  &2\hat\Theta(\partial_v,\partial_v,\partial_v)
  =\partial_v\langle \partial_v,\partial_v \rangle=G_v
 \end{align*}
 hold at the origin $(0,0)$.
 Thus, $\hat \Theta_p$ vanishes if
 and only if
 \eqref{eq:admissible} holds at $p$.
\end{proof}

\begin{definition}\label{def:K2}
 A $C^r$-differentiable positive 
 semi-definite metric $ds^2$ is 
 called a \emph{Kossowski metric} if
 each semi-definite point $p\in M^2$
 of $ds^2$ is admissible 
 and there exists
 a $C^r$-function $\lambda(u,v)$ 
 defined on a local coordinate
 neighborhood $(U;u,v)$ of $p$ 
 such that
 \begin{align}
  \label{eq:k1}
  &  EG-F^2=\lambda^2 \qquad (\mbox{on $U$}), \\
  \label{eq:k2}
  & (\lambda_u(p),\lambda_v(p))\ne (0,0),
 \end{align}
 where $E,F,G$ are $C^r$-functions on $U$
 satisfying \eqref{eq:I}.
\end{definition}

We call such a $\lambda$ the \emph{signed area density function}
of $ds^2$ with respect to the local coordinate 
neighborhood $(U;u,v)$. In fact, if $ds^2$ is positive definite,
then $dA:=|\lambda(u,v)|du\wedge dv$ gives the area element of
the metric $ds^2$.
The function $\lambda$ plays a 
role of an identifier of semi-definite points.
In fact, $\lambda(p)=0$ if and only if $p$
is a semi-definite point.
If $ds^2$ is the first fundamental form of
a frontal $f:M^2\to \R^3$, then 
the function $\lambda$ given in 
\eqref{eq:lambda} coincides
with the signed area density function
of $ds^2$.

As pointed out in the introduction (cf.~Fact \ref{fact:p-back}),
the first fundamental form of
a frontal $f:M^2\to \R^3$
whose singular points are all non-degenerate
is a Kossowski metric.

\begin{lemma}\label{lem:null-add}
 We let $p$ be a semi-definite point of the Kossowski metric
 $ds^2$. 
 Then the null space of $ds^2$ at $p$ is $1$-dimensional.
\end{lemma}
\begin{proof}
 Since $\lambda,F,G,G_u$ and $G_v$ vanish at $p$, 
 twice differentiating 
 the equality $EG-F^2=\lambda^2$ with respect to
 $u$ and $v$, we have
 \[
    2\lambda_u(p)^2=E(p) G_{uu}(p)-2F_u(p)^2,\qquad 
     2\lambda_v(p)^2=E(p) G_{vv}(p)-2F_v(p)^2.
 \]
 If $E(p)=0$ then we have
 $\lambda_u(p)^2+F_u(p)^2=0$ and
 $\lambda_v(p)^2+F_v(p)^2=0$,
 which imply $(\lambda_u(p),\lambda_v(p))=(0,0)$
 contradicting \eqref{eq:k2}.
 So we have $E(p)\ne 0$, that is, $\partial_u$
 is not a null vector. Thus, 
 $\N_p$
 is exactly $1$-dimensional, proving the
 assertion.
\end{proof}

By \eqref{eq:k2}, we can
apply the implicit function theorem
for $\lambda(u,v)=0$, and
find  a $C^r$-regular curve
$\sigma(t)$ $(|t|<\varepsilon)$ in the $uv$-plane 
(called the \emph{characteristic curve}
 or  the \emph{singular curve})
parametrizing the semi-definite set of $ds^2$
such that $\sigma(0)=p$
and $\sigma:(-\epsilon,\epsilon)\to U$ 
is an embedding, where $\epsilon$ is a sufficiently
small positive number.
The following assertion holds:

\begin{proposition}\label{prop:lambda}
 Let $ds^2$ be a $C^r$-differentiable Kossowski metric on $M^2$.
 We let $\lambda:U\to \R$ 
 be a $C^r$-function
 satisfying 
 \eqref{eq:k1} on a connected 
 $C^r$-coordinate neighborhood $(U;u,v)$ of $M^2$.
 Then, the $2$-form
 \begin{equation}\label{eq:dA}
  d\hat A:=\lambda du\wedge dv
 \end{equation}
 does not depend on the choice of
 such local coordinates, up to
 $\pm$-ambiguity, and gives a
 $C^r$-differentiable $2$-form defined on the universal
 covering of $M^2$.
\end{proposition}
\begin{proof}
 Let $(U;u,v)$ be a connected local coordinate neighborhood
 at $p$. Then $ds^2$ has the expression 
 as in \eqref{eq:I}.
 Let $\lambda_1,\lambda_2$ 
 be two signed area density functions
 on $U$ satisfying
 $(\lambda_1)^2=(\lambda_2)^2=EG-F^2$.
 We fix $q\in U$ arbitrarily.
 If $q$ is a regular point, then
 $\lambda_1=\pm \lambda_2$ holds
 on a sufficiently small neighborhood of $W$ of $q$,
 obviously. So we suppose that $q\in U$ is a semi-definite point.
 Since we have observed that the semi-definite 
 set around $q$
 can be parametrized as a regular curve,
 we can take a new local coordinate system 
 $(V;a,b)$ ($V\subset U$)
 centered at $q$ so that the
 $a$-axis is the characteristic  curve.
 Then we have
 $\lambda_1(a,0)=\lambda_2(a,0)=0$. 
 By the division lemma, there exist two $C^r$-function 
 germs $\hat \lambda_1$, $\hat \lambda_2$
 at $(0,0)$  such that
 \[
    \lambda_i(a,b)=b\hat \lambda_i(a,b)\qquad 
    (i=1,2)
 \]
 on $V$.
 In particular, $(\lambda_1)_a(0,0)=(\lambda_2)_a(0,0)=0$
 hold. By \eqref{eq:k2}, we have that
 \[
    0\ne (\lambda_i)_b(0,0)=\hat \lambda_i(0,0) \qquad
          (i=1,2),
 \]
 and
 $\phi:=\lambda_1/\lambda_2=\hat \lambda_1/\hat \lambda_2$
 gives a $C^r$-function defined 
 on a connected neighborhood $W(\subset V)$ of the origin.
 Then we have
 \[
   (\lambda_2)^2=(\lambda_1)^2=(\lambda_2)^2 \phi^2.
 \]
 Since $\lambda_1\ne 0$ except on the $a$-axis,
 $1=\phi^2$
 holds on $W$ by the continuity of $\phi$,
 and that implies $\lambda_1=\pm \lambda_2$ on $W$.
 Since $U$ is connected and $q$ is arbitrarily 
 fixed, $\lambda_1=\lambda_2$ or $\lambda_1=-\lambda_2$ holds 
 on $U$. So we now set $\lambda:=\lambda_1$.
 
 We next prove the second assertion.
 Let $(x,y)$ be another local coordinate system
 on $U$.
 Then 
 \begin{equation}\label{eq:wedge}
  \lambda du\wedge dv=
   \lambda (u_x dx+u_ydy)\wedge (v_x dx+v_ydy)
   =\lambda (u_x v_y-u_yv_x)dx\wedge dy
 \end{equation}
 holds on $U$.
 On the other hand, if we write
 $ds^2=\tilde Edx^2+2\tilde Fdxdy+\tilde G dy^2$,
 then we have that
 \[
     \tilde \lambda^2
         =
          \tilde E\tilde G-\tilde F^2
              =(EG-F^2)(u_x v_y-u_yv_x)^2
            =\lambda^2 (u_x v_y-u_yv_x)^2,
 \]
 and so
 $\pm \lambda (u_x v_y-u_yv_x)$
 gives the area density function 
 with respect to the coordinate
 neighborhood $(U;x,y)$.
 Thus, \eqref{eq:wedge} yields 
 the last assertion.
\end{proof}

\begin{remark}\label{rem:dA}
 The $2$-form $d\hat A$ on $U$ given in \eqref{eq:dA}
 is called a $($local$)$ {\it signed area element}. 
 If $d\hat A$ is well-defined  on  $M^2$, that is,
 if it can be taken to be a $2$-form on $M^2$ 
 so that its restriction to
 each local coordinate neighborhood $(U;u,v)$ gives
 a signed area element of $(U;u,v)$,
 then we say that $ds^2$ is {\it co-orientable}
 on $M^2$.
\end{remark}

Let $p$ be a semi-definite point of a Kossowski metric $ds^2$,
and let $\sigma(t)$ be the characteristic curve 
satisfying $\sigma(0)=p$. 
Then there exists a $C^r$-differentiable non-zero vector field $\eta(t)$
along $\sigma(t)$ which points in the null direction of
the metric $ds^2$.
We call $\eta(t)$ a \emph{null vector field} along 
the characteristic curve $\sigma(t)$.

\begin{definition}
\label{def:a2a3add}
 A semi-definite point $p\in M^2$ of a
 Kossowski metric $ds^2$ is called an 
 \emph{$A_2$ semi-definite point}
 or \emph{semi-definite point of type $A_2$}
 if the derivative $\sigma'(0)$
 of the characteristic  curve at $p$ is linearly independent
 of the null direction $\eta(0)$.
 A semi-definite point $p$ which is not 
 of type $A_2$ is called an  \emph{$A_3$ semi-definite point},
 or \emph{semi-definite point of type $A_3$}
 if  
 \begin{equation}\label{eq:a3}
  \left.\frac{d}{dt}\right|_{t=0}
        \det\left( \sigma'(t),\,\eta(t) \right)\neq0.
 \end{equation}
\end{definition}

\begin{remark}\label{rem:a2a3}
 Cuspidal edges {\rm(}resp.\ swallowtails{\rm)}
 are called $A_2$-singularities 
 {\rm(}resp.\ $A_3$-singularities{\rm)} of wave fronts.
 These points are
 corresponding to $A_2$ semi-definite points 
 {\rm(}resp.\ $A_3$  semi-definite points{\rm)}
 with respect to the induced Kossowski metrics.
 The naming of $A_i$ $(i=2,3)$ points 
 comes from this fact.
\end{remark}

\begin{remark}\label{rem:eta}
 We can extend the null vector field $\eta(t)$
 to be a $C^r$-differentiable vector field $\tilde \eta$ 
 defined on a neighborhood of $p$.
 Then it can be easily checked that
 $p$ is an $A_2$ semi-definite point 
 {\rm(}resp.\ an $A_3$ semi-definite point{\rm)}
 if and only if
 \[
    \lambda_{\tilde \eta}(p)\ne 0
      \qquad (\mbox{resp.\ } 
    \lambda_{\tilde \eta}(p)=0
     \mbox{ and } \lambda_{\tilde\eta\tilde \eta}(p)
          \ne 0),
 \] 
 where 
 $\lambda_{\tilde \eta}:=d\lambda(\tilde \eta)$,
 and
 $\lambda_{\tilde \eta\tilde \eta}:=d\lambda_{\tilde \eta}(\tilde \eta)$.
\end{remark}

We denote by $\Sigma$ the semi-definite set of 
the Kossowski metric $ds^2$ in $M^2$ 
(cf.\ Definition \ref{def:sing-metric}).
Let $K$ be the Gaussian curvature of $ds^2$
defined on $M^2\setminus \Sigma$.
For each sufficiently small local coordinate
system $(U;u,v)$, the signed area element $d\hat A$
is defined (cf.\ Proposition \ref{prop:lambda}).
Then a $2$-form 
\begin{equation}\label{eq:euler0}
   \Omega:=K\,d\hat A
\end{equation}
is defined on $U\setminus \Sigma$, which
can be extended  as a $C^r$-differentiable $2$-form on $U$
(cf.~\cite{Kossowski} and \cite[Theorem 2.15]{HHNSUY}).
We call $\Omega$ the (local) \emph{Euler form} 
associated to $ds^2$ (on $U$).
If $\Omega$ can be extended as a $C^r$-differentiable
$2$-form to $M^2$,
then the integral 
\[
    \frac1{2\pi}\int_{M^2}\Omega
\]
gives
the Euler characteristic of the associated
coherent tangent bundle induced by $ds^2$ when $M^2$
is compact and orientable.
See \cite[Proposition 3.3]{HHNSUY}.

\begin{definition}\label{def:generic}
 A semi-definite point $p\in M^2$ 
 of a Kossowski metric $ds^2$
 is called \emph{parabolic} {\rm(}resp.\ \emph{non-parabolic}{\rm)} 
 if the Euler form $\Omega$ vanishes
 {\rm(}resp.~does not vanish{\rm)} at $p$.
\end{definition}

To prove Fact \ref{fact:front},
we prepare the following lemma:

\begin{lemma}\label{lem:0423}
 Let $f:M^2\to \R^3$ be a $C^r$-differentiable frontal
 and $p\in M^2$ a non-degenerate singular point of $f$.
 Suppose that $p$ is a non-parabolic point with
 respect to the first fundamental form $ds^2$ of $f$,
 then $f$ is a wave front at $p$.
\end{lemma}
\begin{proof}
 We let $p$ be an $A_2$ semi-definite point of $ds^2$.
 As shown in \cite[Page 261]{MSUY}, 
 we can take a local coordinate system
 $(U;u,v)$ centered at $p$ satisfying the following 
 three properties:
 \begin{enumerate}
  \item the $u$-axis coincides with the singular set,
	and $|f_u|=1$ on the $u$-axis,
  \item $f_v(u,0)=\vect{0}$ for each $u$,
  \item $\{f_u,f_{vv},\nu\}$ is an orthonormal frame
	along the $u$-axis.
 \end{enumerate} 
 Then, as shown in \cite[Pages 262--263]{MSUY},
 there is a $C^r$-function $\hat K$ on $U$
 such that 
 \begin{equation}\label{eq:vK}
  \hat K(u,v)=v K(u,v)
 \end{equation}
 on $U\setminus \{v=0\}$, where $K$ is
 the Gaussian curvature of $ds^2$.
 Let $\lambda(u,v)$ be the signed area density
 function on $U$.
 Since $\lambda(u,0)=0$, there exists a
 $C^r$-function $\hat \lambda$ such that
 $\lambda(u,v)=v \hat \lambda(u,v)$.
 Thus, the Euler form can be written as
 \[
    \Omega =K \lambda du \wedge dv
         =\hat K \hat \lambda du \wedge dv.
 \]
 The function $\hat K$ coincides with the same function
 as in \cite[Page 263]{MSUY}. 
 Since $\lambda(u,0)=0$, it holds 
 that $\lambda_u(0,0)=0$.
 By \eqref{eq:k2}, we have 
 $\hat \lambda(0,0)=\lambda_v(0,0)\ne 0$.
 In \cite{MSUY}, the cuspidal curvature $\kappa_c$
 and the product curvature $\kappa_{\Pi}$
 are defined, and we have the following
 (cf.\ \cite[(3.26)]{MSUY})
 \begin{equation}\label{eq:cn1}
  \kappa_\Pi(p)=\kappa_\nu(p) \kappa_c(p).
 \end{equation}
 Moreover, by \cite[(3.25)]{MSUY},
 $\hat K(p)\ne 0$ if and only if $\kappa_\Pi(p)\ne 0$.
 As shown in \cite[Proposition 3.11]{MSUY},
 $\kappa_c(p)\ne 0$ if and only if $f$ is a wave front.
 Since $p$ is non-parabolic point,
 $\kappa_\Pi(p)\ne 0$.
 So $f$ is a wave front at $p$.

 We next consider the case that $p$ is not an $A_2$ 
 semi-definite point.
 As shown in \cite[Page 267]{MSUY}, 
 we can take a local coordinate system
 $(U;u,v)$ centered at $p$ satisfying the following 
 three properties:
 \begin{itemize}
  \item $f_u(0,0)=\vect{0}$,
  \item the $u$-axis is the singular set, and
  \item $|f_v(0,0)|=1$.
 \end{itemize} 
 Using this coordinate system,
 the Euler form satisfies
 $\Omega =\hat K \hat \lambda du \wedge dv$,
 where 
 $\hat K(u,v)=v K(u,v)$, $\hat \lambda$ is a $C^r$-function
 satisfying $\lambda(u,v)=v \hat \lambda(u,v)$
 and $\hat \lambda(0,0)\ne 0$ (cf.\ \cite[Page 270]{MSUY}).
 Moreover, 
 \begin{equation}\label{eq:K1}
  \hat K(u,0)=\mu_\Pi(p)+O(u)
 \end{equation}
 holds (cf.\ \cite[(4.12)]{MSUY}),
 where the normalized cuspidal curvature $\mu_c(p)$
 is defined at \cite[(4.6)]{MSUY} and
 satisfies (cf.\ \cite[(4.10)]{MSUY})
 \begin{equation}\label{eq:K2}
  \mu_\Pi(p)=\kappa_\nu(p)\mu_c(p).
 \end{equation}
 As shown in \cite[Proposition 4.2]{MSUY},
 $\mu_c(p)\ne 0$ if and only if
 $f$ is a wave front at $p$.
 Since $\kappa_\nu(p)\ne 0$,
 \eqref{eq:K1} and \eqref{eq:K2}
 yield that
 $\mu_c(p)\ne 0$ if and 
 only if $\Omega(p)\ne 0$.
 So the fact that $p$ is a non-parabolic point
 implies that $f$ is a wave front at $p$.
\end{proof}
\begin{proof}[Proof of Fact \ref{fact:front}]
 We suppose (1).
 By Lemma \ref{lem:0423},
 (1) implies that $f$ is a wave front.
 We first consider the case that $p$ is an $A_2$ semi-definite
 point.
 Moreover, since $p$ is non-parabolic, as seen
 in the proof of Lemma \ref{lem:0423},
 $\kappa_\Pi(p)\ne 0$ holds. So 
 $\kappa_\Pi(p)=\kappa_\nu(p) \kappa_c(p)$
 implies $\kappa_\nu(p)\ne 0$.
 Therefore, we obtain (2).
 We next consider the case
 that $p$ is not an $A_2$ semi-definite point.
 Since $p$ is non-parabolic, we have
 $\hat K(p)\ne 0$.  
 Then \eqref{eq:K1} yields that $\mu_\Pi(p)\ne 0$.
 As seen in the proof of Lemma \ref{lem:0423},
 the fact that $f$ is a wave front at $p$
 implies $\mu_c(p)\ne 0$, and so
 $\kappa_\nu(p)\ne0$
 by \eqref{eq:K2}. Thus we obtain (2).

 We next suppose (2). Then (3) follows from
 the equivalency of (2) and (3) in \cite[Corollary C]{MSUY}.
 Finally, we suppose (3).
 Since $\nu$ is an immersion  at $p$, $f$ is a wave front at $p$.
 Then the limiting normal curvature of $f$ at $p$ does not
 vanish.
 Then \cite[Theorem A]{MSUY} implies that
 $\Omega(=Kd\hat A)$ does not vanish at $p$.
 So, $p$ is non-parabolic, that is, (1) holds.
\end{proof}

Let $p$ be an $A_2$ semi-definite point of $ds^2$
and $\sigma(t)$ the characteristic  curve 
such that $\sigma(0)=p$.
Since $p$ is of type $A_2$,
the velocity vector $\sigma'(0)$ is not a null vector,
and so we may assume that $t$
is an arc-length parameter of $\sigma$,
that is, $ds^2(\sigma'(t),\sigma'(t))$
is identically equal to $1$.
Then the $2$-form 
\begin{equation}\label{eq:DO}
 \Omega'(p):=\left.\frac{d}{dt}\Omega_{\sigma(t)}
	     \right|_{t=0}
 \in T^*_pM^2\wedge T^*_pM^2
\end{equation}
is defined, which is called the 
\emph{derived  Euler form at $p$} associated with $ds^2$.
The following assertion is an analogue of 
Fact \ref{fact:a2a3}, but we do  not assume
that $f$ is a wave front:

\begin{proposition}\label{prop:Honda}
 Let $f:M^2\to \R^3$ be a $C^r$-frontal
 and $p$ its non-degenerate singular point
 where the limiting normal curvature does not vanish.
 Then 
 \begin{enumerate}
  \item[(o)] $f$ is a wave front at $p$ if and only if $p$ 
	     is non-parabolic {\rm(}i.e.\ $\Omega(p)\ne 0${\rm)}
	     with respect to the first fundamental form
	     $ds^2$ of $f$,
  \item[(i)] $p$ is a cuspidal edge
	     if and only if it is an $A_2$ semi-definite point
	     and $\Omega(p)\ne 0$,
  \item[(ii)] $p$ is a swallowtail
	      if and only if it is an $A_3$ semi-definite point
	      and $\Omega(p)\ne 0$, 
  \item[(iii)] $p$ is a cuspidal cross cap
	       if and only if it is an $A_2$ semi-definite point,
	       $\Omega(p)=0$ and
	       $\Omega'(p)\ne 0$.
 \end{enumerate}
\end{proposition}
\begin{proof}
 We use the same notations as in the proof of
 Lemma \ref{lem:0423}.
 The first assertion (o) follows from the
 equivalency of (1) and (2) of Fact~\ref{fact:front}.
 The assertions (i) and (ii) immediately follow from
 Fact \ref{fact:front} and Fact \ref{fact:a2a3}.
 We next prove (iii).
 Take an $A_2$ semi-definite point $p$.
 The conditions $\Omega(p)=0$ and $\Omega'(p)\ne 0$
 are equivalent to
 the conditions 
 \begin{equation}\label{eq:Ku}
  \hat K(0,0)=0 \mbox{ and } \hat K_u(0,0)\ne 0.
 \end{equation}
 Moreover, by \cite[(3.25)]{MSUY}, 
 \eqref{eq:Ku} is reduced to
 \begin{equation}\label{eq:Pi}
   \kappa_\Pi(p)=0  \mbox{ and } \kappa'_\Pi(p)\ne 0,
 \end{equation}
 where $\kappa_\Pi$ (resp.\ $\kappa'_\Pi$)
 is the product curvature 
 (resp.\ the derivative product curvature) for
 $A_2$ semi-definite points defined in \cite{MSUY}.
 In \cite{MSUY}, 
 the derivative cuspidal curvature $\kappa'_c$
 is also defined, and we have the following
 identity (cf.\ \cite[(3.26)]{MSUY})
 \begin{equation}\label{eq:cn2}
  \kappa'_\Pi(p)=\kappa'_\nu(p) \kappa_c(p)
   +\kappa_\nu(p) \kappa'_c(p).
 \end{equation}
 We let $\sigma(t)$ be a 
 characteristic curve such that $\sigma(0)=p$
 and denote by $\kappa_\nu(t)$
 and $\kappa_c(t)$ the limiting normal
 curvature and the cuspidal curvature
 at $\sigma(t)$, respectively.
 Since $f$ has non-vanishing 
 limiting normal curvature,
 $\kappa_\nu(0)\ne 0$ holds.
 By \eqref{eq:cn1} and \eqref{eq:cn2},
 the condition \eqref{eq:Pi}
 is equivalent to the conditions 
 \begin{equation}\label{eq:P2}
  \kappa_c(0)=0\,\,\mbox{ and }\,\,
   \kappa'_c(0)\ne 0.
 \end{equation}
 On the other hand, the function $\psi_{ccr}(t)$
 defined in \cite[Fact 2.4 (3)]{MSUY}
 satisfies the identity
 $\psi_{ccr}(t)=\kappa_c(t)$,
 as shown in the proof of
 \cite[Proposition 3.11]{MSUY}.
 Since $\psi_{ccr}(0)=\kappa_c(0)$ and $\psi'_{ccr}(0)=\kappa'_c(0)$,
 \eqref{eq:P2} is equivalent to the criterion for
 cuspidal cross caps 
 given in
 \cite[Fact 2.4 (3)]{MSUY}. So we obtain (iii).
\end{proof}

\begin{remark}
 The assertion {\rm (iii)} of Proposition \ref{prop:Honda}
 may not hold if we neglect the assumption
 that the limiting normal curvature of $f$
 does not vanish.
 More precisely, there exists a map germ
 at a cuspidal edge singular point $p$
 satisfying $\Omega(p)=0$ and $\Omega'(p)\ne 0$:
 As shown in \cite{MS}, any germs of cuspidal edges
 are congruent to
 \begin{equation}\label{eq:MS}
  f(u,v)=\biggl(u,a_0(u)+v^2,b_0(u)u^2+b_2(u)uv^2+b_3(u,v)v^3\biggr),
 \end{equation}
 where $b_3(0,0)\ne 0$.
 In this normal form,
 we set
 \[
    a_0(u)=b_2(u)=0, 
         \quad b_3(u,v)=\frac16,
         \quad b_0(u)=\frac{u}2. 
 \]
 Then we obtain a wave front 
 $f(u,v)=\left(u,v^2,{u^3}/2+{v^3}/6\right)$
 having cuspidal edge singularity at $(0,0)$
 such that
 \[
    \kappa_c(0)=1,\quad \kappa'_c(0)=0,\quad
      \kappa_\nu(0)=0,\quad \kappa'_\nu(0)=3.
 \]
 The product curvature $\kappa_\Pi(u)$ 
 and the derivative product curvature
 $\kappa'_\Pi(u)$ satisfy
 $\kappa_\Pi(0)=0$ and $\kappa'_\Pi(0)=3$,
 which yield
 $\Omega(0,0)=0$, $\Omega'(0,0)\ne 0$,
 as seen in the proof 
 of Proposition \ref{prop:Honda}.
\end{remark}

\begin{corollary}\label{cor:CRR}
 Let $ds^2$ be a
 $C^r$-differentiable Kossowski 
 metric $ds^2$, and
 let $p\in M^2$ be an $A_2$ semi-definite 
 point of $ds^2$ satisfying $\Omega(p)=0$
 and $\Omega'(p)\ne 0$. 
 Let  $(U;u,v)$ be a local 
 $C^r$-coordinate system
 centered at $p$ satisfying the properties {\rm (1)--(3)} 
 in the proof of 
 Lemma \ref{lem:0423}.
 Then there exist positive constants
 $\epsilon,\,\, \delta$  such that the sign of the
 Gaussian curvature function $K(u,v)$
 satisfies
 \[
    \sign K(u,v)=\sign(uv \hat K_u(0,0)) 
         \qquad ((u,v)\in C_{\epsilon,\delta}),
 \]
 where $\hat K$ is the function
 defined in \eqref{eq:vK} and
 \[
     C_{\epsilon,\delta}:=
         \{(u,v)\in U\,;\,0<|v|<\delta |u|,\,\,
                |u|<\epsilon
         \}.
 \]
\end{corollary}
\begin{proof}
 We can write 
 \[
    \hat K(u,v)=v K(u,v),\quad
     \hat K(u,v)=\hat K(u,0)+v \phi(u,v),
 \]
 where $\phi(u,v)$ is a $C^r$-function at $p$.
 Moreover, by \eqref{eq:Ku}, we can write
 \[
    \hat K(u,0)=u \psi(u) \qquad (\psi(0)=\hat K_u(0,0)\ne 0),
 \]
 where $\psi(u)$ is a $C^r$-function 
 defined for sufficiently small $|u|$. 
 Without loss of generality, we may assume that
 $\phi(u,v)$ and $\psi(u)$ are defined on 
 a domain 
 \[
    W:=\{(u,v)\,;\, \mbox{$|u|<\epsilon$ and $|v|<\epsilon$}\},
 \]
 where $\epsilon$ is a sufficiently small
 positive number.
 Choosing $W$ so that $\overline{W}\subset U$, 
 we have the expression
 $K(u,v)={u\psi(u)}/{v}+\phi(u,v)$
 on $W\setminus \{v\ne 0\}$.
 Since $\psi(0)\ne 0$ and $\epsilon$ can be
 taken to be arbitrarily small,
 we may assume 
 $|\psi(u)|>m\,(>0)$ and $|\phi(u,v)|<\Delta$ hold on 
 $\overline{W}$ for some constants $m,\Delta$.
 We set
 $\delta:=m/\Delta$.
 If $|u/v|>1/\delta$, we have
 \[
     \left|\frac{u\psi(u)}{v}\right|
           >\frac{m}{\delta}=\Delta>|\phi(u,v)|
 \]
 on $C_{\epsilon,\delta}$.
 So the sign of $K(u,v)$ on $C_{\epsilon,\delta}$
 is equal to that of
 ${u\psi(u)}/{v}$.
\end{proof}

\section{Properties of Kossowski metrics}
\label{sec2}

In this section, we show the existence of a certain
orthogonal coordinate system, which will be applied
to prove Theorems A and B. Using this, we also give a method to
construct Kossowski metrics having
$A_2$ semi-definite points and $A_3$ semi-definite points.

\subsection{K-orthogonal coordinates}
\begin{definition}\label{def:Ev}
 Let $ds^2$ be a $C^r$-differentiable
 Kossowski metric on $M^2$, and take a point $p$ on $M^2$.
 (We also consider the case that $p$ is a regular point.)
 A local coordinate neighborhood 
 $(U;u,v)$ centered at $p$ is called
 a \emph{K-orthogonal coordinate system} if
 \begin{enumerate}
  \item $E=1$ holds along the $u$-axis, 
  \item $F=0$ on $U$, and
  \item $E_v=0$ holds along the characteristic curve $\sigma$
	(when $p$ is a semi-definite point),
 \end{enumerate}
 where we set
 $ds^2=E\,du^2+2F\,du\,dv+G\,dv^2$
 and (3) corresponds to the second condition of \eqref{eq:admissible}.
 In this situation, if we set 
 $\rho:=\sqrt{E}$,
 then the metric $ds^2$ has the following expression
 \begin{equation}\label{cor:normal}
  ds^2=(\rho du)^2+\left(\frac{\lambda dv}{\rho}\right)^2,\qquad
  \rho>0,\qquad
  \rho(u,0)=1,
 \end{equation}
 where $\lambda$ is the signed area density function on $U$. 
\end{definition}

Since $F=G=0$ at a semi-definite point $p$, 
the following assertion trivially holds:
\begin{proposition}\label{prop:null}
 Let $\sigma(t)$ be a
 characteristic curve passing through 
 a semi-definite point $p$ of $ds^2$ and   
 $(u,v)$ a {K}-orthogonal coordinate system
 centered at $p$.
 Then $\partial_v$ belongs to $\N_{\sigma(t)}$
 for each $t$.
 In particular, $\partial_v$ gives a null 
 vector field along $\sigma$.
\end{proposition}

We shall apply the following lemma 
given in Kossowski \cite{Kossowski}
to prove our main theorem:

\begin{lemma}\label{lem:adap_ortho}
 Let $ds^2$ be a $C^r$-differentiable Kossowski metric on $M^2$, 
 and take a point $p$ on $M^2$.
 Let $\gamma$ be a $C^r$-regular curve passing through $p\,\,(=\gamma(0))$
 such that $\gamma'(0)$ is not a null vector of 
 $ds^2$ on $M^2$
 {\rm(}when $p$ is a semi-definite point{\rm)}.
 Then there exists a $C^r$-local coordinate 
 neighborhood $(U;u,v)$
 satisfying the following properties:
 \begin{enumerate}
  \item the $u$-axis corresponds to the curve $\gamma$,
  \item the $u$-curves are orthogonal to the $v$-curves 
	with respect to $ds^2$,
  \item $\partial_v$ points in the null direction
	at each semi-definite point on $U$,
  \item if $ds^2$ and $\gamma$ are real analytic, 
	then so is $(u,v)$.
 \end{enumerate}
\end{lemma}
\begin{proof}
 When $p$ is a regular point, we take
 a $C^r$-differentiable vector field $X_2$ on $U$ such that
 $X_2$ has no zeros on $U$.
 On the other hand, if $p$ 
 is a semi-definite point, we define $X_2$ as follows:
 Let $\sigma$ be the characteristic  curve
 passing through $p$. 
 We take a null vector field $\eta$
 along $\sigma$. We then extend $\eta$ 
 as a $C^r$-differentiable vector field $\tilde \eta$ defined
 on a local coordinate neighborhood $(U;u,v)$ 
 by replacing $U$ with a tubular neighborhood of 
 $\sigma$ in the $uv$-plane.
 We set $X_2:=\tilde \eta$.
 Take a $C^r$-differentiable vector 
 field $X_1$ on $U$ so that
 the curve $\gamma$ is an integral curve
 of $X_1$.
 Since $\gamma'(0)$ is not a null vector,
 we may assume  that the vector fields in the pair
 $(X_1,X_2)$ are linearly independent at each point
 on $U$. By \cite[Lemma B.5.4]{UY-book},
 there exists a $C^r$-differentiable 
 local coordinate system $(x,y)$
 centered at $p$ such that
 $\partial_x,\partial_y$ are 
 proportional to $X_1,X_2$, respectively, and
 the $x$-axis parametrizes $\gamma$.
 We next set
 \[
   Y_1:=\partial_x,\qquad
   Y_2:=-\tilde{F}\,\partial_x+\tilde{E}\,\partial_y,
 \]
 where 
 $ds^2=\tilde E\,dx^2
    +2 \tilde F\,dx\,dy+\tilde G\,dy^2$.
 Then $Y_1$, $Y_2$ are $C^r$-differentiable vector
 fields without zeros
 satisfying $ds^2(Y_1,Y_2)=0$.
 By \cite[Lemma B.5.4]{UY-book} again,
 there exists a new $C^r$-differentiable local coordinate 
 system $(\tilde u,\tilde v)$
 centered at $p$ such that
 $\partial_{\tilde u},\partial_{\tilde v}$ are 
 proportional to $Y_1,Y_2$, respectively, and
 the $\tilde u$-axis parametrizes $\gamma$.
 (In fact, by the proof of \cite[Lemma B.5.4]{UY-book},
 one can check that $(\tilde u,\tilde v)$ is real analytic 
 whenever $(x,y)$ is.)
 Since $\tilde u$-axis corresponds to the $x$-axis
 and $Y_2$ is proportional to $\partial_y$ on the
 characteristic curve $\sigma$, we can conclude that
 $\partial_{\tilde v}$ gives a null vector field
 along $\sigma$.
 Hence, the coordinates $(\tilde u,\tilde v)$ are
 the desired ones.
\end{proof}

\begin{lemma}
\label{lem:adap_ortho2}
 Let $ds^2$ be a $C^r$-differentiable
 Kossowski metric on $M^2$,
 and let $(U;u,v)$ be a $C^r$-local coordinate system such that
 \begin{equation}\label{eq:K000}
  ds^2=Edu^2+(\lambda^2/E) \, dv^2
 \end{equation}
 and $E>0$ on $U$, 
 where $\lambda$ is the signed area density function.
 Then the new $C^r$-local coordinate system $(\tilde u,\tilde v)$
 defined by
 \begin{equation}\label{eq:newcoord}
  \tilde u:=\int_0^u \sqrt{E(t,0)}dt,\qquad 
   \tilde v:=v
 \end{equation}
 gives a K-orthogonal coordinate system on $U$.
\end{lemma}

\begin{proof}
 We can write
 \[
    ds^2=\rho(u,v)^2(\sqrt{E(u,0)}du)^2+\tilde \lambda(u,v)^2 
            \rho(u,v)^{-2}dv^2,   
 \]
 where
 \[
    \tilde \lambda(u,v):=\frac{\lambda(u,v)}{\sqrt{E(u,0)}},\qquad
         \rho(u,v):=\sqrt{\frac{E(u,v)}{E(u,0)}}.
 \]
 By giving the new coordinate system 
 $(\tilde u,\tilde v)$ as in
 \eqref{eq:newcoord},
 and replacing the notation $(\tilde u,\tilde v)$ 
 and $\tilde \lambda$
 by the original $(u,v)$ and $\lambda$, 
 we obtain the expression
 \eqref{cor:normal}.
\end{proof}

\begin{figure}[bht]
 \begin{center}
  \begin{tabular}{c@{\hspace{2em}}c}
{\unitlength 0.1in%
\begin{picture}(16.0000,16.0000)(4.0000,-20.0000)%
\put(11.6000,-12.4000){\makebox(0,0)[rt]{O}}%
\put(11.3000,-4.0000){\makebox(0,0)[rt]{$v$}}%
\put(21.0000,-12.4000){\makebox(0,0)[rt]{$u$}}%
%
\special{pn 8}%
\special{pa 1200 2000}%
\special{pa 1200 400}%
\special{fp}%
\special{sh 1}%
\special{pa 1200 400}%
\special{pa 1180 467}%
\special{pa 1200 453}%
\special{pa 1220 467}%
\special{pa 1200 400}%
\special{fp}%
%
\special{pn 8}%
\special{pa 400 1200}%
\special{pa 2000 1200}%
\special{fp}%
\special{sh 1}%
\special{pa 2000 1200}%
\special{pa 1933 1180}%
\special{pa 1947 1200}%
\special{pa 1933 1220}%
\special{pa 2000 1200}%
\special{fp}%
\special{pn 4}%
\special{pa 400 1200}%
\special{pa 2000 1200}%
\special{fp}%
\put(16.5000,-8.5000){\makebox(0,0){$\Sigma$}}%
%
\special{pn 35}%
\special{pa 400 1200}%
\special{pa 1920 1200}%
\special{fp}%
%
\special{pn 8}%
\special{pa 1350 1140}%
\special{pa 1350 1150}%
\special{fp}%
\special{sh 1}%
\special{pa 1350 1150}%
\special{pa 1370 1083}%
\special{pa 1350 1097}%
\special{pa 1330 1083}%
\special{pa 1350 1150}%
\special{fp}%
%
\special{pn 8}%
\special{ar 1550 1150 200 290 3.1415927 4.7123890}%
\end{picture}}%
  &
{\unitlength 0.1in%
\begin{picture}(16.0000,16.0000)(4.0000,-20.0000)%
\put(11.6000,-12.4000){\makebox(0,0)[rt]{O}}%
\put(11.3000,-4.0000){\makebox(0,0)[rt]{$v$}}%
\put(21.0000,-12.4000){\makebox(0,0)[rt]{$u$}}%
%
\special{pn 8}%
\special{pa 1200 2000}%
\special{pa 1200 400}%
\special{fp}%
\special{sh 1}%
\special{pa 1200 400}%
\special{pa 1180 467}%
\special{pa 1200 453}%
\special{pa 1220 467}%
\special{pa 1200 400}%
\special{fp}%
%
\special{pn 8}%
\special{pa 400 1200}%
\special{pa 2000 1200}%
\special{fp}%
\special{sh 1}%
\special{pa 2000 1200}%
\special{pa 1933 1180}%
\special{pa 1947 1200}%
\special{pa 1933 1220}%
\special{pa 2000 1200}%
\special{fp}%
\special{pn 25}%
\special{pa 1520 400}%
\special{pa 1496 430}%
\special{pa 1493 435}%
\special{pa 1481 450}%
\special{pa 1478 455}%
\special{pa 1466 470}%
\special{pa 1463 475}%
\special{pa 1459 480}%
\special{pa 1456 485}%
\special{pa 1452 490}%
\special{pa 1449 495}%
\special{pa 1445 500}%
\special{pa 1442 505}%
\special{pa 1438 510}%
\special{pa 1435 515}%
\special{pa 1431 520}%
\special{pa 1428 525}%
\special{pa 1424 530}%
\special{pa 1415 545}%
\special{pa 1411 550}%
\special{pa 1402 565}%
\special{pa 1398 570}%
\special{pa 1362 630}%
\special{pa 1360 635}%
\special{pa 1351 650}%
\special{pa 1349 655}%
\special{pa 1340 670}%
\special{pa 1338 675}%
\special{pa 1335 680}%
\special{pa 1333 685}%
\special{pa 1330 690}%
\special{pa 1328 695}%
\special{pa 1325 700}%
\special{pa 1323 705}%
\special{pa 1320 710}%
\special{pa 1318 715}%
\special{pa 1315 720}%
\special{pa 1313 725}%
\special{pa 1310 730}%
\special{pa 1304 745}%
\special{pa 1301 750}%
\special{pa 1295 765}%
\special{pa 1292 770}%
\special{pa 1268 830}%
\special{pa 1267 835}%
\special{pa 1261 850}%
\special{pa 1260 855}%
\special{pa 1254 870}%
\special{pa 1253 875}%
\special{pa 1251 880}%
\special{pa 1250 885}%
\special{pa 1248 890}%
\special{pa 1247 895}%
\special{pa 1245 900}%
\special{pa 1244 905}%
\special{pa 1242 910}%
\special{pa 1241 915}%
\special{pa 1239 920}%
\special{pa 1238 925}%
\special{pa 1236 930}%
\special{pa 1233 945}%
\special{pa 1231 950}%
\special{pa 1228 965}%
\special{pa 1226 970}%
\special{pa 1214 1030}%
\special{pa 1214 1035}%
\special{pa 1211 1050}%
\special{pa 1211 1055}%
\special{pa 1208 1070}%
\special{pa 1208 1075}%
\special{pa 1207 1080}%
\special{pa 1207 1085}%
\special{pa 1206 1090}%
\special{pa 1206 1095}%
\special{pa 1205 1100}%
\special{pa 1205 1105}%
\special{pa 1204 1110}%
\special{pa 1204 1115}%
\special{pa 1203 1120}%
\special{pa 1203 1125}%
\special{pa 1202 1130}%
\special{pa 1202 1145}%
\special{pa 1201 1150}%
\special{pa 1201 1165}%
\special{pa 1200 1170}%
\special{pa 1200 1230}%
\special{pa 1201 1235}%
\special{pa 1201 1250}%
\special{pa 1202 1255}%
\special{pa 1202 1270}%
\special{pa 1203 1275}%
\special{pa 1203 1280}%
\special{pa 1204 1285}%
\special{pa 1204 1290}%
\special{pa 1205 1295}%
\special{pa 1205 1300}%
\special{pa 1206 1305}%
\special{pa 1206 1310}%
\special{pa 1207 1315}%
\special{pa 1207 1320}%
\special{pa 1208 1325}%
\special{pa 1208 1330}%
\special{pa 1211 1345}%
\special{pa 1211 1350}%
\special{pa 1214 1365}%
\special{pa 1214 1370}%
\special{pa 1226 1430}%
\special{pa 1228 1435}%
\special{pa 1231 1450}%
\special{pa 1233 1455}%
\special{pa 1236 1470}%
\special{pa 1238 1475}%
\special{pa 1239 1480}%
\special{pa 1241 1485}%
\special{pa 1242 1490}%
\special{pa 1244 1495}%
\special{pa 1245 1500}%
\special{pa 1247 1505}%
\special{pa 1248 1510}%
\special{pa 1250 1515}%
\special{pa 1251 1520}%
\special{pa 1253 1525}%
\special{pa 1254 1530}%
\special{pa 1260 1545}%
\special{pa 1261 1550}%
\special{pa 1267 1565}%
\special{pa 1268 1570}%
\special{pa 1292 1630}%
\special{pa 1295 1635}%
\special{pa 1301 1650}%
\special{pa 1304 1655}%
\special{pa 1310 1670}%
\special{pa 1313 1675}%
\special{pa 1315 1680}%
\special{pa 1318 1685}%
\special{pa 1320 1690}%
\special{pa 1323 1695}%
\special{pa 1325 1700}%
\special{pa 1328 1705}%
\special{pa 1330 1710}%
\special{pa 1333 1715}%
\special{pa 1335 1720}%
\special{pa 1338 1725}%
\special{pa 1340 1730}%
\special{pa 1349 1745}%
\special{pa 1351 1750}%
\special{pa 1360 1765}%
\special{pa 1362 1770}%
\special{pa 1398 1830}%
\special{pa 1402 1835}%
\special{pa 1411 1850}%
\special{pa 1415 1855}%
\special{pa 1424 1870}%
\special{pa 1428 1875}%
\special{pa 1431 1880}%
\special{pa 1435 1885}%
\special{pa 1438 1890}%
\special{pa 1442 1895}%
\special{pa 1445 1900}%
\special{pa 1449 1905}%
\special{pa 1452 1910}%
\special{pa 1456 1915}%
\special{pa 1459 1920}%
\special{pa 1463 1925}%
\special{pa 1466 1930}%
\special{pa 1478 1945}%
\special{pa 1481 1950}%
\special{pa 1493 1965}%
\special{pa 1496 1970}%
\special{pa 1520 2000}%
\special{fp}%
\put(16.0000,-6.0000){\makebox(0,0){$\Sigma$}}%
\end{picture}}%
  \end{tabular}
 \end{center}
\caption{The coordinates $(u,v)$ at an $A_2$ semi-definite point (left) and 
an $A_3$ semi-definite point (right), where $\Sigma$
is the set of semi-definite points.}
\label{fig:a2a3}
\end{figure}
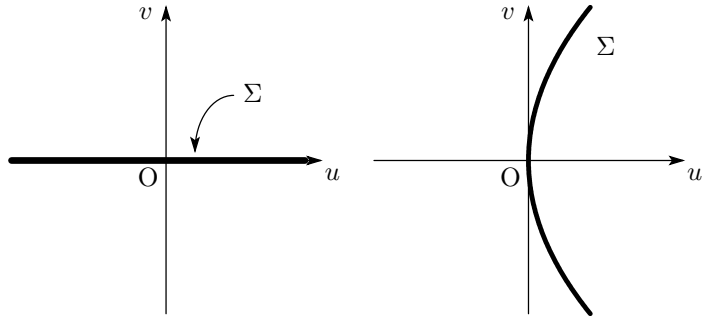

We now prove the following assertion:

\begin{proposition}\label{prop:K-coord}
 Let $ds^2$ be a $C^r$-differentiable
 Kossowski metric on $M^2$, and let
 $\gamma$ be a regular curve
 passing through
 $p$ $(=\gamma(0))\in M^2$ so that $\gamma'(0)$
 is not a null vector when $p$ is a semi-definite point.
 Then there exists 
 a $C^r$-differentiable K-orthogonal coordinate
 system $(U;u,v)$ centered at $p$ such that the
 $u$-axis corresponds to the curve $\gamma$.
 Moreover, 
 if $p$ is an $A_2$ semi-definite point,
 then the $u$-axis can be taken as a
 characteristic curve {\rm(}see Figure \ref{fig:a2a3}{\rm)}.
\end{proposition}
\begin{proof}
 By Lemma \ref{lem:adap_ortho},
 there exists a $C^r$-differentiable 
 orthogonal coordinate system
 $(u,v)$ centered at each point $p\in M^2$ such that
 the metric has the expression
 as in \eqref{eq:K000},
 and the $\tilde u$-axis parametrizes the curve $\gamma$.
 Then we can apply Lemma \ref{lem:adap_ortho2}
 for this coordinate system, and obtain
 the desired K-orthogonal coordinate system. 
 If $p$ is an $A_2$ semi-definite point,
 then we can choose $\gamma$ to be the
 characteristic curve $\sigma$. In this case,
 the $u$-axis parametrizes $\sigma$.
\end{proof}

\subsection{A representation formula for $A_2$ semi-definite points}
In this subsection, we give a representation formula for 
$C^r$-differentiable Kossowski metric germs at $A_2$ semi-definite points.
We fix an $A_2$ semi-definite 
point $p\in M^2$ of a $C^r$-differentiable
Kossowski metric $ds^2$,
and take a $C^r$-differentiable
K-orthogonal coordinate system $(u,v)$ 
(cf.\ Proposition \ref{prop:K-coord})
centered at $p$
with the expression as in
\eqref{cor:normal}.
We may assume the $u$-axis parametrizes the
characteristic curve.
We set
$\omega(u,v):=\log \rho(u,v)$.
Since $\rho(u,0)=1$, we have $\omega(u,0)=0$. So
there exists a $C^r$-function $\omega_1(u,v)$
such that $\omega(u,v)=v \omega_1(u,v)$.
So we can write
\[
   \rho(u,v)=\exp(v \omega_1(u,v)).
\] 
Since $p$ is of type $A_2$, we may assume that the $u$-axis parametrizes
the semi-definite set.
Since $\rho_v(u,0)=0$ holds
(cf.~(3) of Definition \ref{def:Ev}), we have that
$\omega_1(u,0)=0$.
So there exists a $C^r$-function germ $\omega_2(u,v)$
at the origin so that $\omega_1(u,v)=v \omega_2(u,v)$.  
In particular,
$\rho(u,v)=e^{v^2 \omega_2(u,v)}$ holds.
Since the $u$-axis is the semi-definite set,
we have $\lambda(u,0)=0$, and 
there exists a $C^r$-function germ $\hat \lambda(u,v)$
at the origin so that 
$\lambda(u,v)=v \hat \lambda(u,v)$.
Since $(0,0)$ is non-degenerate, we may 
assume $\hat \lambda(0,0)> 0$.
We denote by $C^r_0(\R^2)$ the set of germs of
$C^r$-functions at $(0,0)$ on $\R^2$.
Summarizing the above discussions, we obtain the following assertion.

\begin{theorem}\label{thm:coord}
 Let $h(u,v)$ and $k(u,v)$ be two germs in $C^r_0(\R^2)$.
 Then 
 \[
    ds^2:=\rho^2du^2+ (\rho^{-1}\lambda)^2dv^2
          \qquad \left(\rho:=e^{v^2 h(u,v)},\,\,\,
        \lambda:=v e^{k(u,v)}\right)
 \]
 gives a $C^r$-differentiable
 Kossowski metric germ at an $A_2$ semi-definite point. 
 Conversely, any $C^r$-differentiable Kossowski metric germ with
 $A_2$ semi-definite points is given in this manner.
 Moreover, the Euler form along the semi-definite set
 {\rm(}i.e.\ the $u$-axis{\rm)}
 is given by
 \begin{equation}\label{eq:Omega11}
  \Omega(u,0)
   =
   e^{-k(u,0)} \left(2 h(u,0) k_v(u,0)-3 h_v(u,0)\right)du\wedge dv.
 \end{equation}
\end{theorem}
\begin{proof}
 Let $d\tau^2$ be a Riemannian
 metric which is expressed 
 as $d\tau^2=Edu^2+Gdv^2$. It is well-known that
 the Gaussian curvature $K$ of $d\tau^2$
 is given by (cf. \cite[\S 10]{UY-book})
 \begin{equation}\label{eq:K0000}
  K=-\frac1{\hat e\hat g}\left( 
			  \left(\frac{\hat g_u}{\hat e}\right)_{\!\!u}
			  +
			  \left(\frac{\hat e_v}{\hat g}\right)_{\!\!v}
	\right),
 \end{equation}
 where $\hat e:=\sqrt{E}$ and $\hat g:=\sqrt{G}$.
 Applying this formula for $ds^2$
 at its regular points, we obtain
 \eqref{eq:Omega11}.
 By continuity,
 \eqref{eq:Omega11} holds even at
 semi-definite points of $ds^2$.
\end{proof}

In \cite[Proposition 2.29]{HHNSUY}, we gave another 
representation formula for $A_2$ semi-definite points, which controls 
$E$, $G$ $(:=\lambda^2/E)$ but not $\rho$, $\lambda$.

\subsection{A representation formula for $A_3$ semi-definite points}\
We next consider the case that $p=(0,0)$ is an $A_3$ semi-definite
point
of a Kossowski metric $ds^2$, with the expression as in \eqref{cor:normal}.
This case is not discussed in \cite{HHNSUY}.
We set
\[
   \omega(u,v):=\log \rho(u,v).
\]
Since $\rho=1$ on the $u$-axis,
we  have
$\omega(u,0)=0$.
Since $\partial_v$ gives the tangential direction
of the characteristic curve at $(0,0)$ (cf.\ Figure \ref{fig:a2a3}, right),
the characteristic curve can be expressed as the image 
of a certain graph
$u=g(v)$ satisfying $g(0)=g'(0)=0$.
We set
\[
    \mu(v):=\det\pmt{g'(v) & 0 \\
            1  & 1}=g'(v).
\]
Since $p$ is of type $A_3$,
\eqref{eq:a3} yields that $\mu(0)=0$ and
$\mu'(0)=g''(0)\ne 0$.
Replacing $(u,v)$ by $(-u,v)$ if necessary,
we may assume that $g''(0)>0$
without loss of generality.
Then there exists a $C^r$-function $\hat g(v)$ 
($\hat g(0)> 0$) such that $g(v)=v^2 \hat g(v)$.
Take new coordinates
$\tilde u:=u$ and
$\tilde v:=v \sqrt{\hat g(v)}$, then
the semi-definite set can be expressed as $\tilde u=\tilde v^2$.
So, we may assume that
the parabola $u=v^2$ gives the semi-definite set.
Since $\omega_v(v^2,v)=0$ 
(cf.~(3) of Definition \ref{def:Ev}), 
we have
$\omega_v(u,v)=(u-v^2)a(u,v)$
($a\in C^r_0(\R^2)$). In particular,
$\omega_{vv}(0,0)=0$ 
holds. Since $\omega(u,0)=0$, we have
\[
    \omega(u,v)=\int_0^v (u-w^2)a(u,w)dw.
\]
On the other hand,
since $\lambda(v^2,v)=0$ and $(0,0)$ 
is a non-degenerate semi-definite point, 
we can write
$\lambda=(u-v^2)\hat \lambda(u,v)$ $(\hat \lambda(0,0)\ne 0)$,
where $\hat \lambda\in C^r_0(\R^2)$.
Thus, we obtain the following:

\begin{theorem}\label{thm:coord2}
 Let $h(u,v)$ and $k(u,v)$ be two germs in $C^r_0(\R^2)$.
 Then 
 \begin{align*}
  &ds^2:=\rho^2du^2+ (\rho^{-1}\lambda)^2dv^2,\\
  &\quad \left(\rho(u,v):=\exp\left(\int_0^v (u-w^2)h(u,w)dw\right),
  \,\, \lambda(u,v):=(u-v^2)e^{k(u,v)}\right)
 \end{align*}
 gives a $C^r$-differentiable
 Kossowski metric germ at an $A_3$ semi-definite point. 
 Conversely, any $C^r$-differentiable Kossowski metric germ at
 $A_3$ semi-definite points is given in this manner.
 Moreover, the Euler form at the origin
 is given by
 \begin{equation}\label{eq:O2}
  \Omega(0,0)
   =
   \left[e^{-k(0,0)}\biggl(-h_v(0,0)+h(0,0)k_v(0,0)\biggr)
    -2 e^{k(0,0)} k_u(0,0)\right]
   du\wedge dv.
 \end{equation}
\end{theorem}

The formula \eqref{eq:O2} can be proved
using \eqref{eq:Omega11} and
the fact $\omega_{vv}(0,0)=0$.

\subsection{Distance functions associated with Kossowski metrics}
As an application of the existence of K-orthogonal coordinates,
we investigate 
properties of the distance functions
induced by  Kossowski metrics:
\begin{definition}
 Suppose that $M^2$ is connected.
 Let $ds^2$ be a Kossowski metric on $M^2$,
 and fix two points $p,q\in M^2$.
 We denote by $\P_{p,q}$ the
 set of piecewise smooth arcs combining two points
 $p$ and $q$, and set
 $d_{ds^2}(p,q):=\inf\{L_{ds^2}(\gamma)\,;\gamma\in \P_{p,q} \}$,
 where $L_{ds^2}(\gamma)$ is 
 the length of the arc $\gamma\in \P_{p,q}$
 with respect to $ds^2$,
 that is, 
 \[
     L_{ds^2}(\gamma) := \int_{\gamma} 
           \sqrt{ds^2(\gamma'(t), \gamma'(t))}\,dt.
 \]
 We call $d_{ds^2}:M^2\times M^2\to [0,\infty)$ 
 the \emph{pre-distance function}
 associated with $ds^2$.
\end{definition}

Since $d_{ds^2}(x,y)$ is symmetric with respect to
the variables $x$, $y$ and satisfies the
triangle inequality
by definition, it gives a distance function if
and only if $d_{ds^2}(x,y)=0$ implies $x=y$.

\begin{definition}
 Let $ds^2$ be a Kossowski metric on $M^2$.
 A semi-definite point $p\in M^2$
 is called a \emph{peak}
 if there exists a neighborhood $U$ of $p$
 such that the semi-definite points on
 $U\setminus \{p\}$ consists only of $A_2$ semi-definite points.
\end{definition}

For example, $A_2$ or $A_3$ semi-definite points are peaks.
On the other hand, 
$ds^2_0=du^2+u^2 dv^2$ 
gives a Kossowski metric whose semi-definite set coincides
with the $v$-axis.
Each point of the $v$-axis is not a peak, since the null-direction
$\partial_v$ gives the tangential direction 
of the $v$-axis as its characteristic curve.

\begin{remark}
 In \cite{SUY_annals}, 
 a \lq\lq peak singularity'' on wave fronts is
 defined. Suppose that $ds^2$ is the first fundamental form 
 of a wave front $f$. 
 Let $p\in M^2$ be a non-degenerate  singular point of
 $f$. 
 Then 
 there exists a neighborhood $U$ of
 $p\in M^2$ such that the restriction of $ds^2$ 
 on $U$ is a Kossowski metric.
 Moreover, $p$ is a peak with respect
 to $ds^2$ if and only if $p$
 is at most non-degenerate peak singular point of $f$. 
 This fact immediately follows
 from the definition of peaks  of
 $f$. 
\end{remark}

We show the following assertion:
\begin{theorem}\label{thm:induced-m}
 Let $ds^2$ be a $C^r$-differentiable
 Kossowski metric on $M^2$ 
 whose semi-definite points
 consist only of peaks. Then the 
 pre-distance function 
 associated with $ds^2$ gives a distance function
 which is compatible with the topology of $M^2$.
\end{theorem}
We set
\begin{equation}\label{eq:Bp}
 B_{ds^2}(p,r):=
  \left\{q \in M^2\,;\,
   d_{ds^2}(p,q)<r\right\}
  \qquad (p\in M^2,\,\, r>0).
\end{equation}
To prove Theorem \ref{thm:induced-m},
we prepare the following four lemmas:

\begin{lemma}\label{lemma:induced-1}
 Let $ds^2$ be a $C^r$-differentiable
 Kossowski metric and $d\tau^2$ a 
 $C^\infty$-differentiable
 Riemannian metric on $M^2$
 such that $ds^2<d\tau^2$ on $M^2$
 $($that is,
 $ds^2(\vect{v},\vect{v})< d\tau^2(\vect{v},\vect{v})$
 holds for each $\vect{v}\,(\ne \vect{0}))$.
 Then there exists $\epsilon>0$
 such that 
 \[
    B_{d\tau^2}(p,r)\subset B_{ds^2}(p,r)
 \]
 holds for $r\in (0,\epsilon)$.
\end{lemma}
\begin{proof}
 Since $ds^2<d\tau^2$, $L_{ds^2}(\gamma)\le L_{d\tau^2}(\gamma)$ holds
 for each path $\gamma\in \P_{p,q}$ 
 between two points $p,q\in M^2$.
 So we obtain
 $d_{ds^2}(p,q)\le d_{d\tau^2}(p,q)$
 and
 \[
    B_{d\tau^2}(p,r)
          =\{q\,;\, d_{d\tau^2}(p,q)<r\}
         \subset \{q\,;\, d_{ds^2}(p,q)<r\}
        =B_{ds^2}(p,r).
 \]
\end{proof}

\begin{lemma}\label{lemma:induced-2}
 Let $ds^2$ be a $C^r$-differentiable
 Kossowski metric on $M^2$ and $p\in M^2$
 an $A_2$ semi-definite point. 
 Suppose that 
 $U$  is an open neighborhood of $p$.
 Then there exists $\delta>0$
 such that $B_{ds^2}(p,\delta)\subset U$.
\end{lemma}
\begin{proof}
 We can take a K-orthogonal coordinate
 neighborhood $(V;u,v)$ centered at $p$
 satisfying the following properties
 (cf.\ Theorem \ref{thm:coord}):
 \begin{enumerate}
  \item the closure $\overline{V}$ of $V$ is a subset of $U$,
  \item there exist $C^r$-functions
	$E$ and $G$ on $\overline{V}$ such that
	$ds^2=E\,du^2+G\,dv^2$,
  \item $E>0$ on $\overline{V}$,
  \item the set of semi-definite points of $ds^2$ coincides with the $u$-axis.
 \end{enumerate}
 We set
 $\Omega:=\{(u,v)\in V\,;\, |u|\le \epsilon_1, \,\, |v|\le \epsilon_2 \}$
 for sufficiently small
 $\epsilon_i$ ($i=1,2$) and
 \[
     S:=\{(u,v)\in \Omega\,;\, |v|\ge \epsilon_2/2 \}. 
 \]
 Then
 \[
    m_1:=\min_{\Omega}(\sqrt{E}),\qquad
    m_2:=\min_{S}(\sqrt{G})
 \]
 are positive.
 Suppose  $q=(u_0,v_0)$ lies on the boundary of
 $\Omega$,
 and take $\gamma\in \P_{p,q}$.
 If $|u_0|=\epsilon_1$, the path $\gamma$ travels
 horizontally across   the left or right half of $\Omega$, 
 and so one can easily show that 
 $L(\gamma)>m_1\epsilon_1$. 
 Similarly if $|v_0|=\epsilon_2$
 then $\gamma$ travels vertically across  one 
 of the closed rectangular sub-domains of $S$, and so  
 $L(\gamma)>m_2\epsilon_2/2$.
 Thus, if we set
 \[
      m_0:=\min\biggl(m_1\epsilon_1,\,\,\frac{m_2\epsilon_2}2\biggr),
 \]
 then 
 $B_{ds^2}(p,m_0)\subset \Omega \subset U$,
 proving the assertion.
\end{proof}

\begin{lemma}\label{lem:A2thm}
 Let $ds^2$ be a $C^r$-differentiable
 Kossowski metric on $M^2$ whose semi-definite points
 are all of type $A_2$. Then the 
 pre-distance function 
 associated with $ds^2$ gives a distance function
 which is compatible with the topology of $M^2$.
\end{lemma}

\begin{proof}
 We take two distinct points $p,q\in M^2$.
 To prove $d_{ds^2}$ is a distance function,
 it is sufficient to show that $p\ne q$ implies
 $d_{ds^2}(p,q)\ne 0$.
 Since $M^2$ is a Hausdorff space,
 we can take a  local coordinate neighborhood 
 $(U;u,v)$ of $p$ satisfying
 $q\not\in U$.
 By Lemma \ref{lemma:induced-2},
 there exists $r>0$
 such that $B_{ds^2}(p,r)\subset U$.
 Then $d_{ds^2}(p,q)\ge r$ holds.
 So $d_{ds^2}$ is a distance function. 

 We next show that $d_{ds^2}$ is compatible with
 the topology of $M^2$.
 We fix a  point $p$ of $M^2$ arbitrarily,
 and take a local coordinate neighborhood 
 $(U;u,v)$ centered at $p$
 so that $\overline{U}$ is compact.
 Let $d\tau^2$ be the canonical Euclidean metric on 
 the $uv$-plane. 
 Suppose that $p$ is a regular point of $ds^2$.
 Then it can be easily checked that
 for each $r>0$, there exists $r'>0$
 such that
 $B_{ds^2}(p,r')$
 (resp.\ $B_{d\tau^2}(p,r')$)
 is a subset of
 $B_{d\tau^2}(p,r)$
 (resp. $B_{ds^2}(p,r)$).
 So we consider the case that 
 $p$ is a semi-definite point of $ds^2$.
 By Lemma \ref{lemma:induced-2},
 we have
 \[
       B_{ds^2}(p,r)\subset U.
 \]
 For sufficiently small $r>0$,
 we can take a positive constant
 $m$ such that $ds^2<m \, d\tau^2$ on $\overline{U}$.
 We set $M^2=U$ and apply
 Lemma \ref{lemma:induced-1}.
 Then
 \[
    B_{d\tau^2}(p,r/\sqrt{m})\subset B_{ds^2}(p,r)(\subset U)
 \]
 holds for sufficiently small $r>0$.
 On the other hand, applying Lemma \ref{lemma:induced-2}
 again, there exists $r'>0$
 such that $B_{ds^2}(p,r')\subset B_{d\tau^2}(p,r)$.
 Thus, the topology induced by $d_{ds^2}$ is
 the same as that of $M^2$ as a manifold. 
\end{proof}

\begin{lemma}\label{lemma:induced-a3}
 Let $ds^2$ be a $C^r$-differentiable
 Kossowski metric on $M^2$, and let $p\in M^2$ be
 a peak. 
 Suppose that 
 $U$  is an open neighborhood of $p$.
 Then there exists $\delta>0$
 such that $B_{ds^2}(p,\delta)\subset U$.
\end{lemma}

\begin{proof}
 We can take a local coordinate
 neighborhood $(V;u,v)$ centered at $p$
 such that
 the closure $\overline{V}$ of $V$ is a subset of $U$.
 Take a sufficiently small $\epsilon>0$, and set  
 \[
     \Omega:=\{q\in V\,;\, |q|\le \epsilon  \},
 \]
 where $|(u,v)|:=\sqrt{u^2+v^2}$.
 Consider the subset $S$ of $\Omega$ defined by
 \[
   S:=\{q\in \Omega\,;\, |q|\ge \epsilon/2 \},
 \]
 and set
 \[
    m:=\inf\{d_{ds^2}(q_1,q_2)\,;\, |q_1|=\epsilon,\,\,|q_2|=\epsilon/2\}.
 \]
 Since $U\setminus \{q\in \Omega\,;\, |q|\le\epsilon/3\}$ 
 admits only regular points or $A_2$ semi-definite points,
 Lemma \ref{lem:A2thm} yields that $m>0$.
 Suppose that $q=(u_0,v_0)$ lies on the boundary
 of $\Omega$,
 and take $\gamma\in \P_{p,q}$.
 Since the path $\gamma$ travels
 across  $S$, we have
 $L_{ds^2}(\gamma)>m$. 
 Thus, we have
 $B_{ds^2}(p,m)\subset \Omega \subset U$,
 proving the assertion.
\end{proof}
\begin{proof}[Proof of Theorem \ref{thm:induced-m}]
 If we use Lemma
 \ref{lemma:induced-a3}
 instead of 
 Lemma \ref{lemma:induced-2},
 the same argument
 as in the proof of Lemma \ref{lem:A2thm}
 gives the proof of
 Theorem \ref{thm:induced-m}. 
\end{proof}

\section{Coherent tangent bundles induced by Kossowski metrics}
\label{sec3}

In this section, we deduce the partial differential equation
given in Kossowski \cite{Kossowski},
using the fact (shown in \cite{HHNSUY}) 
that a Kossowski metric induces an associated vector 
bundle with a metric and a connection, 
called a ``coherent tangent bundle''.

\subsection{Fundamental theorem for frontals}
Let $\E$ be  a vector bundle of
rank $2$ over a $2$-manifold $M^2$, and
$\inner{\,}{\,}$ an inner product on $\E$.
We let $\nabla$ be a connection on $\E$
which is compatible with respect to the
inner product.
If a vector bundle homomorphism
$\phi:TM^2 \to \E$
which induces the identity map on $M^2$
satisfies the identity
\begin{equation}\label{eq:Codazzi}
   \nabla_X \varphi (Y) - \nabla_Y \varphi (X) =\varphi([X,Y]) 
  \qquad(X,Y\in \X^r),
\end{equation}
then we call $(\E,\inner{\,}{\,},\nabla,\phi)$
a \emph{coherent tangent bundle} over $M^2$,
where $\X^r$ is the set of $C^r$-vector fields on $M^2$.
(This definition can be generalized for $n$-dimensional 
manifolds, cf.\ \cite{SUY_msj}.)
In this situation, the pull-back metric of $\inner{\,}{\,}$
via $\phi$,
\[
    ds^2:=\phi^*\inner{\,}{\,}
\]
is induced, which is called the \emph{first fundamental form}
of $\phi$. 
A point $p$ where $\phi_p:T_pM^2\to \E_p$ 
has a non-trivial kernel corresponds to 
a semi-definite point of $ds^2$.

\begin{definition}\label{def:coherent-isom}
 Two coherent tangent bundles on $M^2$
 \[ 
     (\E_1,\inner{~}{~}_1,\nabla^1,\varphi_1),\qquad
      (\E_2,\inner{~}{~}_2,\nabla^2,\varphi_2)
 \]
 are said to be \emph{isomorphic}
 if there exists a bundle isomorphism
 $\iota\colon{}\E_1\to \E_2$
 satisfying the following three conditions:
 \begin{itemize}
  \item $\varphi_2=\iota\circ\varphi_1$,  
  \item $\iota$ preserves the inner products, that is,
	for each $p\in M^2$
	and for each $\xi$, $\eta\in (\E_1)_{p}$,
	$\inner{\xi}{\eta}_1=\inner{\iota(\xi)}{\iota(\eta)}_2$ 
	holds,
  \item for each
	$\vect{v}\in T_{p}M^2$
	and for each section $\xi$ of $\E_1$,
	$\iota(\nabla^1_{\vect{v}}\xi) = \nabla^2_{\vect{v}}\iota(\xi)$ holds.
 \end{itemize}
 In this situation, $\iota$ is called
 an \emph{isomorphism} between coherent tangent bundles.
\end{definition}

The following assertion holds:
\begin{fact}[{\cite[Theorem 3.1]{HHNSUY}}]
\label{fact:jitugen}
 Let $M^2$ be an oriented $C^\omega$-differentiable $2$-manifold 
 and $ds^2$ a $C^\omega$-differentiable Kossowski metric on $M^2$.
 Then there exists a unique $C^\omega$-differentiable
 coherent tangent bundle
 $(\E,\inner{\,}{\,},\, \nabla,\phi)$
 up to isomorphisms of coherent tangent bundles
 such that the induced metric $\varphi^*\inner{\,}{\,}$
 coincides with $ds^2$.
 Moreover, $\E$ is orientable if and only if $ds^2$ is
 co-orientable {\rm(}see Remark \ref{rem:dA}{\rm)}. 
\end{fact}
\begin{remark}
 $\E:=TM^2$ can be considered
 as a coherent tangent bundle
 if $\inner{\,}{\,}$ is a Riemannian metric, 
 $\phi$ is the identity map, and
 $\nabla$ is the Levi-Civita connection.
\end{remark}

\begin{remark}
 Fact \ref{fact:jitugen}
 was applied in  \cite{HHNSUY} to prove 
 two Gauss-Bonnet type formulas for Kossowski metrics.
 Moreover, Kossowski metrics can be defined on higher 
 dimensional manifolds, and
 this fact was generalized for arbitrary dimension
 {\rm(}see {\cite[Theorem 7.9]{SUY_msj}}{\rm)}.
\end{remark}

\begin{definition}[Frontal bundles]
 Suppose that there are two bundle homomorphisms
 $\varphi,\,\psi : TM^2 \rightarrow (\E, \inner{\,}{\,}, \nabla)$
 such that each of them induces a structure of 
 a coherent tangent bundle
 on $(\E, \inner{\,}{\,}, \nabla)$.
 If they satisfy the following compatibility condition
 \begin{equation}\label{eq:symmetric}
  \inner{\varphi(X)}{\psi(Y)}=\inner{\varphi(Y)}{\psi(X)}
  \qquad(X,Y\in \X^r),
 \end{equation}
 then  $(\E, \inner{\,}{\,}, \nabla,\phi,\psi)$
 is called a \emph{frontal bundle}.
\end{definition}
\begin{example}
 Let $f:M^2\to \R^3$ be a frontal,
 and let $\nu:M^2\to S^2$ be its unit normal vector field.
 Then 
 \[
    \E_f:=\{(p,\vect{w})\in M^2\times T\R^3\,;\, 
        \mbox{$\vect{w}\in T_p\R^3$ 
               is perpendicular to $\nu_p$}\}
 \]
 has a structure of a vector bundle 
 of rank $2$ over $M^2$.
 The inner product $\inner{\,}{\,}$ is induced from
 the canonical  inner product of $\R^3$.
 Moreover, taking the tangential component of 
 the Levi-Civita connection of $\R^3$, $\E_f$ has
 a connection $\nabla^f$ which is compatible with 
 the metric $\inner{\,}{\,}$.
 Then the two bundle homomorphisms defined by
 \[
    \phi_f:TM^2\ni \vect{v} 
         \mapsto (\pi(\vect{v}),df(\vect{v}))\in \E_f
 \]
 and
 \[
    \psi_\nu:TM^2\ni \vect{v} \mapsto 
           (\pi(\vect{v}),d\nu(\vect{v}))\in \E_f
 \]
 give a structure of frontal bundle.
 where $\pi:TM^2\to M^2$ is the canonical
 projection.
 We call $(\E_f,\inner{\,}{\,},\nabla^f,\phi_f,\psi_\nu)$ 
 the \emph{frontal bundle induced by} $f$.
 The condition \eqref{eq:Codazzi}
 for $\phi_f$ 
 follows from the fact that $\nabla^f$ can be identified
 with the Levi-Civita connection of $M^2\setminus \Sigma_f$,
 where $\Sigma_f$ is the singular set of $f$.
 On the other hand,
 the condition \eqref{eq:Codazzi}
 for $\psi_\nu$ 
 follows from the fact that $f$ satisfies
 the Codazzi equation on $M^2\setminus \Sigma_f$
 (see \cite[Example 2.2]{SUY_kodai} for details).
\end{example}

Let
$(\E,\inner{\,}{\,},\nabla,\phi)$ 
be a coherent tangent bundle over $M^2$.
We fix a local coordinate neighborhood $(U;u,v)$
on $M^2$ with $U$ chosen so that there is also 
an orthonormal frame field $(\vect{e}_1,\vect{e}_2)$ of
$\E$ on $U$.
Such a $5$-tuple $(U;u,v,\vect{e}_1,\vect{e}_2)$ is called
a \emph{local orthonormal trivialization} of $\E$.
Since
$\nabla$ is compatible with respect to the inner product
$\inner{\,}{\,}$,  
for such a $5$-tuple,
there exists a $1$-form $\theta$ defined on $U$
such that
\begin{equation}\label{eq:D1}
 \nabla_{\vect{v}} \vect{e}_1=-\theta(\vect{v})\vect{e}_2,\quad
  \nabla_{\vect{v}} \vect{e}_2=\theta(\vect{v})\vect{e}_1
  \qquad (\vect{v}\in TU)
\end{equation}
and
$d\theta=K d\hat A$
hold
on the set of regular points on $U$
(cf.\ \cite[(13.15)]{UY-book}).
By continuity, 
$d\theta=\Omega$
holds on $U$, where $\Omega$ is the Euler form
of $ds^2$ (cf.\ \eqref{eq:euler0}).
The following assertion was proved in \cite[Section 2]{SUY_kodai},
which plays a role to realize a given Kossowski metric as
the first fundamental form of a frontal.

\begin{theorem}
\label{fact:realization}
 Let $(\E,\, \inner{\,}{\,},\, \nabla,\phi,\psi)$
 be a $C^\omega$-frontal bundle over a simply-connected
 $C^\omega$-local coordinate neighborhood $(U;u,v)$ of $M^2$. 
 Suppose that the induced metric
 $ds^2:=\phi^*\inner{\,}{\,}$ is a Kossowski 
 metric having the expression 
 as in \eqref{cor:normal}. 
 We fix a point $p\in U$ arbitrarily.
 Suppose that
 \begin{equation}\label{eq:Gauss}
  d\theta=\det
   \pmt{
   A & B \\
   C & D}
    du\wedge dv
 \end{equation}
 holds for the local orthonormal 
 trivialization $(\vect{e}_1,\vect{e}_2)$
 on $(U;u,v)$, 
 where
 \begin{align}\label{eq:ABCD*}
  &A:=\inner{\psi(\partial_u)}{\vect{e}_1},\quad
  C:=\inner{\psi(\partial_u)}{\vect{e}_2},\nonumber \\
  &B:=\inner{\psi(\partial_v)}{\vect{e}_1},\quad
  D:=\inner{\psi(\partial_v)}{\vect{e}_2}.
 \end{align}
 Then there exists a unique
 quadruple $(f,\hat{\vect{e}}_1,\hat{\vect{e}}_2,\nu)$ 
 consisting of a
 $C^\omega$-frontal and 
 an orthonormal frame field along $f$
 such that
 \begin{enumerate}
  \item $f$ is a frontal and $\nu$ is a unit normal
	vector field of $f$,
  \item $f_u\cdot \hat{\vect{e}}_j$ 
	and ${f_v}\cdot \hat{\vect{e}}_j$ 
	$(j=1,2)$ coincide with
	$\inner{\phi(\partial_u)}{\vect{e}_j}$
	and
	$\inner{\phi(\partial_v)}{\vect{e}_j}$,
	respectively,
  \item $\nu_u\cdot \hat{\vect{e}}_j$ 
	and $\nu_v\cdot \hat{\vect{e}}_j$ 
	$(j=1,2)$ coincide with
	$\inner{\psi(\partial_u)}{\vect{e}_j}$
	and
	$\inner{\psi(\partial_v)}{\vect{e}_j}$,
	respectively,
  \item $f(p)=\vect{0}$, and  
	$(\hat{\vect{e}}_1,\hat{\vect{e}}_2,\nu)$ is the identity matrix
	at $p$.
 \end{enumerate}
\end{theorem}

\begin{proof}
 Consider the system of partial differential equations as follows:
 \begin{align}\label{pde:1}
  f_u&=\sum_{j=1,2}\inner{\phi(\partial_u)}{{\vect{e}}_j}\hat{\vect{e}}_j, \quad
  f_v=\sum_{j=1,2}\inner{\phi(\partial_v)}{{\vect{e}}_j}\hat{\vect{e}}_j, \\
  \label{pde:2}
  (\hat{\vect{e}}_i)_u&=
  \sum_{j=1,2}
    \inner{\nabla_{\partial_u} \vect{e}_i}{{\vect{e}}_j}\hat{\vect{e}}_j,
     \,\,\,
  (\hat{\vect{e}}_i)_v=
  \sum_{j=1,2}
    \inner{\nabla_{\partial_v} \vect{e}_i}{{\vect{e}}_j}\hat{\vect{e}}_j
  \quad (i=1,2),\\
  \label{pde:3}
  \nu_u&=A \hat{\vect{e}}_1+C \hat{\vect{e}}_2,\quad
  \nu_v=B \hat{\vect{e}}_1+D \hat{\vect{e}}_2.
 \end{align}
 As shown in \cite[Section 2]{SUY_kodai}, the integrability condition of
 this system follows from the fact that
 $(\E, \inner{\,}{\,}, \nabla,\phi,\psi)$
 is a $C^\omega$-frontal bundle
 satisfying \eqref{eq:Gauss}.
 So we obtain the assertion.
\end{proof}

\begin{remark}
 Theorem \ref{fact:realization} corresponds to the 
 fundamental theorem of surface
 theory for frontals, which is described for a
 local coordinate neighborhood. 
 A global version of this assertion is given in
 \cite[Theorem 2.7]{SUY_kodai}.
\end{remark}

Let $(U;u,v)$ be  
a K-orthogonal coordinate system 
as in  \eqref{cor:normal}
(cf.\ Proposition \ref{prop:K-coord})
with respect to
a Kossowski metric $ds^2$. 
By Fact \ref{fact:jitugen},
there is a bundle homomorphism
\begin{equation}\label{eq:TU}
 \phi:TU \rightarrow \E
\end{equation}
such that $(\E, \inner{\,}{\,}, \nabla,\phi)$
is a coherent tangent bundle satisfying
$ds^2=\phi^*\inner{\,}{\,}$ on $U$. 
Then
\begin{equation}\label{eq:e1}
  \vect{e}_1:=\frac{1}{\rho}\,\varphi(\partial_u)
  \qquad (\rho:=\sqrt{E})
\end{equation}
gives a unit vector at each fiber of $\E$ on $U$.
We then take a local section $\vect{e}_2$ of $\E$
on $U$ such that $(\vect{e}_1,\vect{e}_2)$ consists of an
orthonormal frame field of $\E$ on $U$.
There are $C^\omega$-functions $k,h$ on $U$ such that
$\varphi(\partial_v)=k\vect{e}_1+h \vect{e}_2$.
Since $(u,v)$ is a K-orthogonal coordinate system,
$k$ vanishes identically.
Moreover, we have 
\[
  \lambda^2 E^{-1}=G=\inner{\varphi(\partial_v)}{\varphi(\partial_v)}
      =h^2, 
\]
and we obtain
$h={\lambda}/{\sqrt{E}}$
by replacing $\vect{e}_2$ by $-\vect{e}_2$ if necessary. 
So it holds that
\begin{equation}\label{eq:e2}
  \varphi(\partial_v)
     =\frac{\lambda}{\rho}\,\vect{e}_2.
\end{equation}
We set
\begin{equation}\label{eq:n1}
  \nabla_{\partial_u}\vect{e}_1 = \alpha\, \vect{e}_2,\qquad
  \nabla_{\partial_v}\vect{e}_1 = \beta\, \vect{e}_2,
\end{equation}
where $\alpha,\beta$ are $C^\omega$-functions on $U$.
Then $\theta=-\alpha du-\beta dv$
gives a $1$-form on $U$ satisfying
\eqref{eq:D1}, that is,
\[
  \nabla_{\partial_u}\vect{e}_2 = -\alpha\, \vect{e}_1,\qquad
  \nabla_{\partial_v}\vect{e}_2 = -\beta\, \vect{e}_1.  
\]
\begin{proposition}
 The functions $\alpha$ and $\beta$ are given by
 \begin{equation}\label{eq:Christoffel}
  \alpha = -\frac{ E_v}{2  \lambda },\qquad
   \beta = \frac{2E\lambda_u-\lambda E_u}{2E^2}.
 \end{equation}
\end{proposition}
\begin{proof}
 By \eqref{eq:Codazzi}, we have that
 \begin{align*}
  0
  &=\nabla_{\partial_u} \phi(\partial_v)-\nabla_{\partial_v} \phi(\partial_u)
  =\nabla_{\partial_u} \left(\frac{\lambda}{\rho}\vect{e}_2\right)
  -\nabla_{\partial_v} (\rho\vect{e}_1)\\
  &=
  \left(\frac{\lambda}{\rho}\right)_u \vect{e}_2+
  \frac{\lambda}{\rho}(-\alpha \vect{e}_1)
  -\rho_v\vect{e}_1-\rho(\beta \vect{e}_2)\\
  &=
  \left(\left(\frac{\lambda}{\rho}\right)_u-\rho \beta\right)\vect{e}_2
  -
  \left(
  \frac{\lambda\alpha}{\rho}+\rho_v
  \right)\vect{e}_1.
 \end{align*}
 Thus, we obtain
 \[
    \left(\frac{\lambda}{\rho}\right)_u-\rho \beta=0,\qquad
    \frac{\lambda\alpha}{\rho}+\rho_v=0.
 \]
 Since $\rho=\sqrt{E}$,
 these are equivalent to 
 \eqref{eq:Christoffel}.
\end{proof}

It is well-known that the Gaussian curvature $K$
defined at regular points of $ds^2$ on $U$ satisfies 
(cf.\ \cite[(13.15)]{UY-book})
\[
  K \lambda du\wedge dv=d\theta=(\alpha_v-\beta_u)du\wedge dv.
\]
So it holds that
\begin{equation}\label{eq:chK}
 K \lambda=\alpha_v-\beta_u.
\end{equation}

\begin{remark}
 By  \eqref{eq:admissible}, 
 $E_v$ vanishes on the semi-definite set
 of the metric. So $E_v/\lambda$ is a $C^\omega$-function on $U$.
\end{remark}

We would like to find a new bundle homomorphism
$\psi:TU\to \E$
so that $(\E, \inner{\,}{\,}, \nabla,\phi,\psi)$
is a frontal bundle.
For this purpose, let $A,B,C,D$ be unknown functions
satisfying
\begin{equation}\label{eq:d-nu}
  \psi(\partial_u) = A\,\vect{e}_1+ C\,\vect{e}_2,\qquad
  \psi(\partial_v) = B\,\vect{e}_1+ D\,\vect{e}_2.
\end{equation}

\begin{proposition}
 \label{lem:compatibility}
 In this setting, the following assertions hold:
 \begin{enumerate}
  \item $(\E, \inner{\,}{\,}, \nabla,\phi,\psi)$
	on a simply-connected domain 
	$(U;u,v)$ is a frontal bundle if and only if
	$A,B,C,D$ satisfy
        \begin{align}
	 \label{eq:Codazzi_1}\tag{Cod}
	 &B_u-A_v  =  \alpha D - \beta C,\qquad
	 D_u-C_v  =  \beta A - \alpha B,\\
	 \label{eq:symm}\tag{Symm}
	 &EB = \lambda C.
	\end{align}
  \item The integrability condition
	\eqref{eq:Gauss} is equivalent to the
	condition
	\begin{equation}\label{eq:egregium}\tag{Gauss}
	 AD - BC =  \check{K},
	\end{equation}
	where
	$\check{K}:=\alpha_v-\beta_u(=K\lambda) 
	$ {\rm(}cf.\ \eqref{eq:chK}{\rm)}.
  \item There exists a
	$C^\omega$-differentiable frontal $f:U\to \R^3$ 
	whose first fundamental form
	is $ds^2$ if there exist $C^\omega$-functions
	$A,B,C,D$ satisfying
	\eqref{eq:Codazzi_1}, 
	\eqref{eq:symm}, \eqref{eq:egregium}
	and \eqref{eq:ABCD*}.
 \end{enumerate}
\end{proposition}

\begin{proof}
 The map $\psi$ as in
 \eqref{eq:d-nu} satisfies
 \eqref{eq:Codazzi_1}
 if and only if
 $\nabla_{\partial_u}\psi(\partial_v)
 =\nabla_{\partial_v}\psi(\partial_u)$.
 On the other hand,
 \eqref{eq:symmetric} (resp. \eqref{eq:Gauss}) is
 equivalent to
 the condition \eqref{eq:symm}
 (resp.\ \eqref{eq:egregium}).
 So we can apply
 Theorem \ref{fact:realization}, and obtain the assertion.
\end{proof}
 
\begin{remark}
 The system of partial differential equations $($PDE$)$
 given by
 \eqref{eq:Codazzi_1},
 \eqref{eq:symm} and
 \eqref{eq:egregium}
 is the same as in \cite[(5) in Page 108]{Kossowski}.
 However, there is a sign typographical error
 in \cite[(5)]{Kossowski}, and the above
 PDE corrects it.
\end{remark}

Since $\lambda$ vanishes along the semi-definite set,
\eqref{eq:symm} yields the following:

\begin{corollary}\label{cor:B}
 The function $B$ as in
 Proposition \ref{lem:compatibility}
 vanishes identically along
 the semi-definite set.
\end{corollary}

\subsection{Second fundamental data of frontal maps}

We now fix a $C^\omega$-frontal 
$f:U\to \R^3$
defined on a simply-connected domain $U$
such that its first fundamental form $ds^2$ is a
$C^\omega$-differentiable Kossowski metric.
We let $\nu:U\to \R^3$ be a unit normal vector field
along $f$ and fix a point $p\in U$. Then 
there exists a $C^\omega$-differentiable
K-orthogonal coordinate system
$(u,v)$ centered at the origin $p$ by 
Proposition \ref{prop:K-coord}.
Without loss of generality, we may assume that
$(u,v)$ is defined on $U$. 
Then we have the expression \eqref{cor:normal}.
So we set
\begin{equation}\label{eq:h-e1}
 \hat{\vect{e}}_1:=\frac{1}{\rho}f_u,
\end{equation}
which is a unit vector field defined on $U$.
We then define a unit vector field
\[
   \hat{\vect{e}}_2:=\nu \times \hat{\vect{e}}_1.
\]
By definition, $f_v$ is a scalar multiplication of $\hat{\vect{e}}_2$.
Since
\[
  f_v \cdot \hat{\vect{e}}_2=\det(f_v,\nu \times \hat{\vect{e}}_1)=
  \frac{1}{\rho}\det(f_v,\nu,f_u)=
  \frac{1}{\rho}\det(f_u,f_v,\nu)=\frac{\lambda}{\rho},
\]
we have
\begin{equation}\label{eq:h-e2}
    f_v=\frac{\lambda}{\rho}\hat{\vect{e}}_2.
\end{equation}
We set
\begin{equation}\label{eq:abcd}
  A:=\nu_u \cdot \hat{\vect{e}}_1,\quad
  B:=\nu_v \cdot \hat{\vect{e}}_1,\quad 
  C:=\nu_u \cdot \hat{\vect{e}}_2,\quad
  D:=\nu_v \cdot \hat{\vect{e}}_2. 
\end{equation}
We call $(A,B,C,D)$ the {\it second fundamental data} of $f$.

As discussed in the previous subsection,
the first fundamental form $ds^2$
of $f$ induces a bundle homomorphism
(cf.\ \eqref{eq:TU})
$\phi:TU \rightarrow \E$
such that $(\E, \inner{\,}{\,}, \nabla,\phi)$
is a coherent tangent bundle on $U$ satisfying
$ds^2=\phi^*\inner{\,}{\,}$.
Then we obtain an orthogonal trivialization
$(U;u,v,\vect{e}_1,\vect{e}_2)$ 
of the coherent tangent bundle $(\E, \inner{\,}{\,}, \nabla,\phi)$
by \eqref{eq:e1} and \eqref{eq:e2}.
We then define $\psi:TU\to \E$ by
\[
  \psi(\partial_u)=A \vect{e}_1+C \vect{e}_2,\quad
  \psi(\partial_v)=B \vect{e}_1+D \vect{e}_2.
\]

\begin{lemma}\label{lem:000}
 The second fundamental data of $f$ 
 satisfies \eqref{eq:Codazzi_1}, \eqref{eq:symm} and \eqref{eq:egregium}.
 In particular, $\psi$ is a bundle homomorphism
 such that $(\E,\inner{\,}{\,}, \nabla,\phi,\psi)$
 is a frontal bundle satisfying \eqref{eq:Gauss}.
\end{lemma}
\begin{proof}
 Since $ds^2$ is a Kossowski metric, the regular set 
 $R_f$ of $f$ is open dense in $U$.
 By \eqref{eq:abcd}, $(A,B,C,D)$ satisfies
 \[
    \nu_u=A \hat{\vect{e}}_1+C \hat{\vect{e}}_2,\quad
    \nu_v=B \hat{\vect{e}}_1+D \hat{\vect{e}}_2
 \]
 on $U$. Since $f$ is an immersion on $R_f$, 
 \eqref{eq:Codazzi_1}, \eqref{eq:symm} and \eqref{eq:egregium}
 hold on $R_f$.
 By the continuity, these formulas also hold on $U$.
\end{proof}

By Proposition \ref{lem:compatibility}
and Lemma \ref{lem:000},
there exists a $C^\omega$-frontal 
$f_0:U\to \R^3$
whose first fundamental form
is $ds^2$ such that $(A,B,C,D)$ is the second fundamental data
of $f_0$ with respect to a unit normal vector 
field $\nu_0$.
Then we can prove the following:
\begin{proposition}\label{prop:000}
 $f$ is strongly congruent to $f_0$.
\end{proposition}
\begin{proof}
 The pair $(f_0,\nu_0)$ satisfies
 \eqref{pde:1}, \eqref{pde:2} and \eqref{pde:3} on $U$.
 On the other hand,
 the pair $(f,\nu)$ satisfies
 \eqref{pde:1}, \eqref{pde:2} and \eqref{pde:3}
 on $R_f$. By the  continuity, 
 they also hold  on $U$.
 Hence the two pairs $(f_0,\nu_0)$ and $(f,\nu)$
 satisfy the same system 
 of partial differential equations
 \eqref{pde:1}, \eqref{pde:2} and
 \eqref{pde:3}.
 So the uniqueness of the solution, 
 $f$ is strongly congruent to $f_0$.
\end{proof}

\begin{corollary}\label{cor:Need2}
 Let $f,f':U\to \R^3$ be two frontal maps
 with the same first fundamental form $ds^2$. 
 Then $f,g$ are strongly congruent if
 and only if they have the same second fundamental
 data up to $\pm$-multiplication.
\end{corollary}

\begin{proof}
 Let $\nu$ (resp.\ $\nu'$) be the unit normal vector
 of $f$ (resp.\ $f'$).
 Suppose that $f'$ is strongly congruent to $f$,
 then $f'$ has the same second fundamental
 data as $f$ by replacing $f'$ by $-f'$ if necessary.

 On the other hand, suppose that 
 $f$ and $f'$ have the same second fundamental
 data up to a $\pm$-multiplication.
 By replacing $f$ by $-f$, we may assume that
 $f$ and $f'$ have the same second fundamental data.
 Then by definition,
 $f_0$ and $f'_0$ have the same second fundamental data.
 By Proposition \ref{lem:compatibility},
 $f_0$ is strongly congruent to $f'_0$.  
 By Proposition \ref{prop:000},
 $f$ is strongly congruent to $f'$.  
 So we obtain the conclusion.
\end{proof}

\section{Isometric realizations of Kossowski metrics}
\label{sec4}

\subsection{Proof of Theorem \ref{thm:main0}}
To prove our main results,
we need to apply the following
Cauchy-Kowalevski
theorem:
\begin{fact}[cf.\ \cite{KP}]\label{fact:CK}
 Let $F^i(u,v,z^1,z^2,w^1,w^2)$ $(i=1,2)$ 
 be two real analytic functions defined 
 on a domain $\D$ of $\R^6$, and let
 $\omega^i:(-\epsilon,\epsilon)\to \R$ $(i=1,2)$
 be  real analytic functions so that
 \[
    (u,0,\omega^1(u),\omega^2(u),(\omega^1)'(u),(\omega^2)'(u))\in \D
       \qquad (|u|<\epsilon),
 \]
 where $\epsilon>0$ is a sufficiently small number.
 Then there exists a unique real analytic map 
 $\phi=(\phi^1,\phi^2):U\to \R^2$ defined on a 
 neighborhood $U$ of the origin of the $uv$-plane
 such that
 \[
   \phi^i_v(u,v)
      =F^i(u,v,\phi^1(u,v),\phi^2(u,v),\phi^1_u(u,v),\phi^2_u(u,v))
        \qquad (i=1,2),
 \]
 and $\phi^i(u,0)=\omega^i(u)$ for $i=1,2$.
\end{fact}

We fix a Kossowski metric $ds^2$ defined on 
a real analytic $M^2$,
and take a point $p\in M^2$.
Let $\gamma$ be a $C^\omega$-regular curve in $M^2$
such that $\gamma(0)=p$
and $\gamma'(0)$ is not a null direction.
By Proposition \ref{prop:K-coord},
we can take 
real analytic K-orthogonal coordinates  $(u,v)$ centered at $p$.
Then $\gamma(u)=(u,0)$ holds.
Let $\rho, \lambda$ be given as in
\eqref{cor:normal} and
$\alpha,\beta$ defined by
\eqref{eq:Christoffel}.
Since $ds^2$ is real analytic,
the four functions $\rho, \lambda,\alpha,\beta$ are
all real analytic on $U$.
Then we can consider the system of partial differential equations
\eqref{eq:Codazzi_1},
\eqref{eq:symm} and
\eqref{eq:egregium}
with unknown functions $A,B,C,D$.
We now assume 
$A(0,0)\ne 0$.
By \eqref{eq:symm}  and
\eqref{eq:egregium},
we can set
\begin{equation}\label{eq:BD}
  B:=\frac{\lambda C}{E},\qquad
  D:=\frac{\check{K}+BC}{A}=\frac{E\check{K}+\lambda C^2}{EA},
\end{equation}
and substituting them into
\eqref{eq:Codazzi_1},
we obtain the following normal form of a PDE
\begin{align}\label{eq:standard_form}
  A_v &= \left(\frac{\lambda C}{E}\right)_u 
            - \alpha \frac{E\check{K}+\lambda C^2}{EA} + \beta C,
 \nonumber \\
  C_v &= \left(\frac{E\check{K}+\lambda C^2}{EA}\right)_u 
            -  \beta A + \alpha \frac{\lambda C}{E} 
\end{align}
with unknown functions $A$ and $C$.

We fix two function germs $a(u)$ and $c(u)$
defined at $u=0$ so that $a(0)\ne 0$.
By applying Fact \ref{fact:CK}, 
there exist $A,C$ satisfying 
\eqref{eq:standard_form}
and
\begin{equation}\label{eq:AC-ini}
 A(u,0)=a(u),\qquad C(u,0)=c(u)
\end{equation}
defined on a
certain neighborhood $V(\subset U)$ of the origin.

Then, 
by Proposition \ref{lem:compatibility},
there exists a frontal
$f:=f_{a,c}:V\to \R^3$ with unit normal vector $\nu$
whose first fundamental form is $ds^2$
and the second fundamental data is $(A,B,C,D)$.
In particular, if we set
(cf.\ \eqref{eq:e1} and \eqref{eq:e2})
\begin{equation}\label{eq:e1e2}
  \hat{\vect{e}}_1:=\frac{f_u}{\rho}, \qquad
  \hat{\vect{e}}_2:=\frac{\rho f_v}{\lambda}
  \qquad
  (\rho:=\sqrt{E}),
\end{equation}
then
$d s^2$ has a local expression as in \eqref{cor:normal},
and
$(\hat{\vect{e}}_1,\hat{\vect{e}}_2,\nu)$
gives an orthonormal frame
along $f$ so that
\begin{equation}\label{eq:fufv0}
  f_u=\rho\,\hat{\vect{e}}_1,\quad
  f_v=\lambda \rho^{-1}\,\hat{\vect{e}}_2,
  \quad \nu=\hat{\vect{e}}_1\times \hat{\vect{e}}_2.
\end{equation}
Moreover, it holds that
\begin{equation}\label{eq:nunv0}
  \nu_u=A \hat{\vect{e}}_1+C \hat{\vect{e}}_2,\qquad
  \nu_v=B \hat{\vect{e}}_1+D \hat{\vect{e}}_2.
\end{equation}
By our choice of $(u,v)$, 
$\gamma(u)=(u,0)$ holds.
Since $\gamma'(0)$ is not a null vector,
$\hat \gamma(u):=f\circ \gamma(u)=f(u,0)$
gives a regular space curve, and 
so the normal curvature function $\kappa_n$
(cf.\ \eqref{eq:kn1}) of $f$ along the curve $\gamma(u)=(u,0)$
can be considered as follows.

\begin{proposition}\label{prop:43}
 Let $\kappa_n(u)$ be the normal curvature function 
 of $f$ along the curve $\gamma(u)=(u,0)$.
 Then 
 \begin{equation}\label{eq:a}
  a(u)=-\kappa_n(u)
 \end{equation}
 holds.
 Moreover, if $p$ is an $A_2$ semi-definite point and 
 the 
 curve $\gamma$ is
 the characteristic  curve, then $\kappa_n(u)$ coincides
 with the limiting normal curvature function along 
 $\gamma(u)$.
\end{proposition}

\begin{proof}
 Let $\hat{\vect{e}}_1,\hat{\vect{e}}_2$ be
 vector fields given in \eqref{eq:e1e2}.
 By \eqref{eq:nunv0}, we have
 \[
    (\hat{\vect{e}}_1)_u\cdot \nu=-
           \hat{\vect{e}}_1\cdot \nu_u=-A.
 \]
 Together with \eqref{eq:n1}, we have
 $(\hat{\vect{e}}_1)_u=\alpha \hat{\vect{e}}_2-A\nu$.
 Similarly, we have
 \[
   (\hat{\vect{e}}_1)_v=\beta \hat{\vect{e}}_2-B\nu, \quad
   (\hat{\vect{e}}_2)_u=-\alpha \hat{\vect{e}}_1-C\nu, \quad
   (\hat{\vect{e}}_2)_v=-\beta \hat{\vect{e}}_1-D\nu.
 \]
 Differentiating \eqref{eq:fufv0} using the above formulas, 
 we have that
 \begin{alignat}{1}
  \label{eq:fuu}
  f_{uu}
           &=\rho_u\,\hat{\vect{e}}_1+\rho\alpha \,\hat{\vect{e}}_2 
-\rho A \,\nu,\\
  \label{eq:fuv}
  f_{uv}
     &=\rho_v\,\hat{\vect{e}}_1+\rho\beta \,\hat{\vect{e}}_2 -\rho B \,\nu,\\
  f_{vv}
     &=  -\lambda \rho^{-1}\beta \,\hat{\vect{e}}_1 + (\lambda
  \rho^{-1})_v 
            \,\hat{\vect{e}}_2  -  \lambda \rho^{-1}D \,\nu.
 \end{alignat}
 Since $\rho(u,0)=1$, we have
 \begin{equation}\label{eq:kn}
   \kappa_n(u)=\frac{f_{uu}(u,0)\cdot \nu(u,0)}{\rho(u,0)^2}=
     -A(u,0)=-a(u).
 \end{equation}
 If $\gamma$ is a characteristic curve parametrizing $A_2$ 
 semi-definite points,
 then by \cite[(2.2) and (2.3)]{MSUY}
 the limiting normal curvature $\kappa_\nu(u)$
 defined by \eqref{eq:kn0} 
 coincides with $\kappa_n(u)$
 defined by \eqref{eq:kn1}.
\end{proof}
\begin{proof}{Proof of Theorem~\ref{thm:main0}}
 The existence of isometric realization of $ds^2$ 
 as a frontal map has been proved.
 So we need to prove the remaining properties.
 Since $a(0)\ne 0$, 
 \eqref{eq:a} implies that $\kappa_n(0)\ne 0$.
 By \eqref{eq:kn0}, this $\kappa_n(0)$ is equal to
 the limiting normal curvature of $p$.
 So the first assertion of Theorem A is obtained.
 Assertions (1)--(4) follow immediately from 
 Proposition \ref{prop:Honda}.
\end{proof}
\subsection{Proofs of Theorem \ref{thm:main} and 
    Corollaries \ref{cor:main}, \ref{cor:germ-A2}, \ref{cor:cw}}

We next give  preparations to prove 
Theorem B. As shown in the previous subsection, 
for given real analytic function germs $a(u)$, $c(u)$
at $u=0$ satisfying $a(0)\ne 0$, 
we constructed a real analytic frontal
$f_{a,c}:V\to \R^3$
whose first fundamental form was $ds^2$
satisfying \eqref{eq:fufv0}
and \eqref{eq:nunv0}.
The congruence class of $f_{a,c}$ is determined 
from the initial data $a(u)$, $c(u)$ as follows.

\begin{lemma}
\label{rmk:U}
 Two frontals $f_{a,c}$ and 
 $f_{\tilde a, \tilde c}$
 are mutually strongly congruent if
 and only if
 $(\tilde a,\tilde c)=\epsilon (a, c)$
 for some $\epsilon \in \{1,-1\}$.
\end{lemma}

\begin{proof}
 The initial data $(a,c)$ and $(-a,-c)$
 induce the solutions $(A,C)$ and $(-A,-C)$
 of \eqref{eq:standard_form}, respectively.
 By \eqref{eq:BD}, $(A,C)$ and $(-A,-C)$
 induce quadruples $(A,B,C,D)$ and $(-A,-B,-C,-D)$
 satisfying
 \eqref{eq:Codazzi_1},
 \eqref{eq:symm} and
 \eqref{eq:egregium}.
 By Proposition~\ref{lem:compatibility}
 and Corollary \ref{cor:Need2},
 $(A,B,C,D)$ and $(-A,-B,-C,-D)$ induce
 the same frontal up to strong congruence.
\end{proof}

We next compute the geodesic curvature of $\hat \gamma(u):=f(u,0)$,
where $f:=f_{a,c}$.

\begin{proposition}
\label{prop:geod}
 Let $\kappa_g(u)$ be the geodesic curvature function 
 of $f$ along the curve $\gamma(u)=(u,0)$
 {\rm(}$\kappa_g(u)$ coincides with the singular curvature
 when the $u$-axis parametrizes the characteristic curve{\rm)}.
 By adjusting the sign of $\kappa_g(u)$, it holds that
 $\alpha(u,0)=\kappa_g(u)$.
\end{proposition}

\begin{proof}
 Since $\rho=1$ on the $u$-axis, we have
 \[
    \kappa_g=\frac{f_{uu}\cdot \hat{\vect{e}}_2}{f_u\cdot f_u}
        =\frac{\alpha}{\rho}=\alpha
 \]
 along the $u$-axis, proving the assertion.
\end{proof}

Propositions \ref{prop:43} and \ref{prop:geod}
lead to the following:
\begin{corollary}
\label{cor:kappa}
 The curvature function $\kappa(u)$
 of $\hat\gamma(u)=f(u,0)$ as a regular space curve
 is given by
 \begin{equation}\label{eq:kappa000}
  \kappa(u)=\sqrt{\alpha(u,0)^2+a(u)^2}.
 \end{equation}
\end{corollary}

We next compute the torsion function of $\hat \gamma$.

\begin{proposition}
\label{prop:C}
 The torsion function $\mu(u)$ of 
 $\hat \gamma(u)$ satisfies
 \begin{equation}\label{eq:C}
  \mu(u)=-c(u)+\frac{a(u) \alpha_u(u,0)-\alpha(u,0) a'(u)}{\kappa(u)^2},
 \end{equation}
 where $\kappa(u)$ is the curvature function of $\hat\gamma(u)$.
\end{proposition}

\begin{proof}
 It is well-known that
 \[
   \kappa(u)^2\mu(u)=\frac{\det(\hat \gamma'(u),
   \hat \gamma''(u),\hat \gamma'''(u))}{|\hat \gamma'(u)|^6}.
 \]
 Since $|\hat \gamma'(u)|=\rho(u,0)=1$, we have
 \begin{align*}
  \kappa(u)^2\mu(u)&=
  \left.\det(f_u,f_{uu},f_{uuu})\right|_{(u,v)=(u,0)}\\
  &=\left.\det(\hat{\vect{e}}_1,\alpha\hat{ \vect{e}}_2-a(u)\nu,f_{uuu}
  )\right|_{(u,v)=(u,0)}.
 \end{align*}

 So it is sufficient to compute
 $f_{uuu}$
 modulo a functional multiplication of $\hat {\vect{e}}_1$.
 Using the fact that $\rho=1$ along the $u$-axis, we have
 \[
   f_{uuu} \equiv
       \bigl(\alpha_u-AC
          \bigr)\hat {\vect{e}}_2
     +
    \bigr(-A_u-\alpha C\bigr)\nu \qquad \mod \hat{\vect{e}}_1.
 \]
 Then 
 \begin{align*}
  &\left.
  \kappa(u)^2\mu(u)
  =\det(f_u,f_{uu},f_{uuu})\right|_{(u,v)=(u,0)} \\
  &\phantom{\kappa(u)^2\mu(u)A}
  =\alpha_u(u,0)a(u)-\alpha(u,0) a'(u)-c(u)(a(u)^2+\alpha(u,0)^2).
 \end{align*}
 Since $a(u)^2+\alpha(u,0)^2=\kappa(u)^2$ holds
 by Corollary \ref{cor:kappa},
 we obtain the conclusion.
\end{proof}
\begin{proof}[Proof of Theorem \ref{thm:main}]
 We set
 \begin{equation}\label{eq:ac}
   a(u):=-e^{\omega(u)},\qquad 
   c(u):=-\mu(u)+\frac{a(u) \alpha_u(u,0)-
   \alpha(u,0) a'(u)}{a(u)^2+\alpha(u,0)^2},
 \end{equation}
 as the initial values of $A$ and $C$.
 Then we obtain a frontal $f=f_{a,c}$ whose first fundamental form is $ds^2$.
 Moreover, $\hat \gamma(u)$ has the normal curvature function $e^{\omega(u)}$
 and the torsion function $\mu(u)$.
 By this construction, the first, second, and third assertions are obvious.
 From now on, we prove the last assertion.
 Since $a(0)\ne 0$ and $\rho(0,0)=1$, we have 
 \[
   0\ne \nu_u(p)\cdot \hat{\vect{e}}_1(p)=\nu_u(p)\cdot 
           f_u(p)=-\nu(p)\cdot f_{uu}(p).
 \]
 Since $(u,v)$ is adjusted at $p$, we can conclude that
 $p$ is a non-$\nu$-flat point of $f$ (cf. Definition~\ref{def:gen}). 

 By replacing the unit normal vector field $\nu$ to $-\nu$,
 the sign of the limiting normal curvature is reversed.
 Hence, by \eqref{eq:kappa000} and
 Lemma \ref{rmk:U}, 
 the possibilities of $a(u)$ as the initial value of $A$
 are 
 $a(u)=e^{\omega(u)}$ or $-e^{\omega(u)}$.
 Since the case $a(u)=-e^{\omega(u)}$ produces $f=f_{a,c}$, and
 so, the other possibility is the case that
 \[
    \tilde a(u):=e^{\omega(u)}(=-a(u))
 \]
 as the initial data of $A$. In this case
 \[
   \tilde c(u):=
   -\mu(u)-\frac{\tilde a(u) \alpha_u(u,0)-\alpha(u,0) \tilde
   a'(u)}{\tilde a(u)^2+\alpha(u,0)^2}=-2\mu(u)-c(u)
 \]
 must be the initial data of $C$. 
 Since $\tilde a(u)=-a(u)$,
 Lemma \ref{rmk:U}
 yields that $f_{\tilde a,\tilde c}$
 is strongly congruent to $f_{a,c}$ if  
 $\tilde c(u)=-c(u)$,
 that is, $\mu(u)$ vanishes identically.
 So there are at most two possibilities of the congruence class for $f$, 
 unless $\mu$ is identically zero. 
 If $p$ is a regular point, then $f$ must be an
 immersion since $ds^2$ is positive definite.
 If $p$ is a non-parabolic singular point, 
 then $\nu$ must be an immersion, and
 $f$ is a wave front germ.
\end{proof}

\begin{proof}[Proof of Corollary \ref{cor:main}]
 As seen in the above proof of Theorem \ref{thm:main},
 $f$ is uniquely determined by the initial data $a(u)$ and $c(u)$,
 and depends on them real analytically.
 Since $a$ and $c$ can be written explicitly 
 in terms of $\omega$ and $\mu$ as in 
 \eqref{eq:ac}, we obtain the assertion.
\end{proof}
\begin{proof}[Proof of Corollary \ref{cor:germ-A2}]
 Let $\kappa_\nu(t)$ and $\kappa_s(t)$
 be the limiting normal curvature and
 the singular curvature (cf.~\cite{HHNSUY})
 of the characteristic curve $\sigma(t)$, respectively.
 Since we may assume $\kappa_\nu(0)>0$, there exists a real 
 analytic function $\omega(t)$ such that
 $\kappa_\nu(t)=e^{\omega(t)}$.
 Then the curvature function $\kappa(t)$ of $\hat \sigma(t)$
 as a regular space curve is given by
 \begin{equation}\label{eq:kappa0}
  \kappa(t)=\sqrt{\kappa_\nu(t)^2+\kappa_s(t)^2}.
 \end{equation}
 By Fact \ref{fact:k_n=k_nu}, $\kappa_\nu(t)$ coincides with
 the normal curvature of $f$ along $\sigma(t)$.
 Let $\mu(t)$ be the torsion function of the space curve $\Gamma(t)$.
 Since $p$ is an $A_2$ semi-definite point, $\sigma'(0)$ is not a
 null vector, and
 so, by Theorem~\ref{thm:main},
 there exists a real analytic frontal germ $g$ at $p$ whose normal
 curvature and torsion along $g\circ \sigma(t)$ coincide with
 $\kappa_\nu(t)$ (cf.\ \eqref{eq:kappa0})
 and $\mu(t)$, respectively.
 By this construction, the curvature function 
 of the regular space curve $g\circ \sigma(t)$
 equals $\kappa(t)$, and $\mu(t)$ gives
 the torsion function of $g\circ \sigma(t)$.
 Since $g\circ \sigma(t)$ and $\Gamma(t)$ are
 parametrized by an arc-length parameter and
 have the same curvature and torsion,
 we can conclude that $\Gamma(t)=g\circ \sigma(t)$.
 The property that $g$ has a cuspidal edge or a 
 cuspidal cross cap
 at $p$ depends on the induced Kossowski
 metric of $f$ (cf.\ Theorem \ref{thm:main0}).
 Thus, if $p$ is a cuspidal edge (resp. cuspidal cross cap)
 with respect to $f$, then this is so  with respect to $g$, 
 too. 
 By the last  assertion of Theorem \ref{thm:main},
 the number of
 strong congruence classes of $g$ is at most two.
\end{proof}
\begin{proof}[Proof of Corollary \ref{cor:cw}]
 Since $f_0$ and $f_1$ are isometric (cf. Definition \ref{def:isom}), 
 there exists a local diffeomorphism $\phi$ such that
 $g_0:=f_0$ and $g_1:=f_1\circ \phi$ induce the same 
 Kossowski metric $ds^2$. Let $p$ be a semi-definite point of  $ds^2$,
 and $(u,v)$ a K-orthogonal coordinate system centered at $p$.
 We fix a unit normal vector field $\nu_i(u,v)$ of $f_i$,
 and then four real analytic functions 
 \[
    A_i(u,v),\,\, B_i(u,v),\,\, C_i(u,v),\,\, D_i(u,v)
 \]
 are determined. Then $(A_i,C_i)$ ($i=0,1$)
 can be considered as a solution of 
 \eqref{eq:standard_form} which induces $f_i$.
 We then set
 \[
    a_i(u)=A_i(u,0),\quad
    c_i(u)=C_i(u,0)\qquad (i=0,1).
 \]
 The sign of the limiting normal curvature of
 the characteristic curve of $f_i$ with respect to $\nu_i$
 is equal to the sign of $-a_i(u)$.
 So, as long as considering isometric deformations 
 with non-vanishing limiting normal curvature,
 the sign of $\kappa_\nu$ does not change.
 So, to deform $(f_0,\nu_0)$ to $(f_1,\nu_1)$
 continuously,
 we must adjust the sign of $\pm \nu_i$ ($i=0,1$).
 Replacing
 the sign of $\nu_i$
 of $f_i$ for each $i=0,1$ if necessary, 
 we may assume that
 $a_0(u),a_1(u)<0$,
 where we used the fact that 
 the limiting normal curvature of $f_i$ does not vanish.
 For each $s\in [0,1]$, we set
 \[
    a_s(u):=(1-s)a_0(u)+sa_1(u)(<0),\qquad
    c_s(u):=(1-s)c_0(u)+sc_1(u).
 \]
 Then, there exists a unique solution $(\tilde A_s(u,v),\tilde C_s(u,v))$
 of \eqref{eq:standard_form}
 satisfying 
 \[
    \tilde A_s(u,0)=a_s(u),\qquad
    \tilde C_s(u,0)=c_s(u).
 \]
 Then we obtain a family of frontals 
 $g_s:V(\subset U)\to \R^3$ ($0\le s\le 1$),
 interpolating between $g_0$ and $g_1$,
 that have the common first fundamental form $ds^2$.
 Since $a_s(u)<0$,
 the limiting normal curvature of each $g_s$ is positive.
 Then
 $f_t:=g_t\circ \phi$ $(0\le t\le 1)$,
 gives the desired deformation.
 The second assertion follows from the
 fact that $p$ is a cuspidal edge, a
 swallowtail or
 a cuspidal cross cap is determined by the properties of 
 the Kossowski metric $ds^2$ (cf.\ Proposition~\ref{prop:Honda}). 
\end{proof}

\begin{remark}\label{rmk:Tf}
 Let $f$ be a real analytic frontal  
 germ with singularities
 whose limiting normal curvature
 does not vanish.
 Let $T$ be an orientation
 reversing isometry of $\R^3$.
 Then $T\circ f$ has the same first fundamental form
 as $f$, but it is not trivial that $f$ can be 
 isometrically deformed into $T\circ f$.
 Let $\nu$ be the unit normal vector of
 $f$ such that $\kappa_\nu>0$ along the 
 characteristic curve.
 Then $T\circ f$ has the same limiting normal
 curvature $\kappa_\nu$ as $f$ if we
 choose $-dT\circ \nu$ as a normal vector field of 
 $T\circ f$.
 So the above proof yields that
 the pair $(f,\nu)$ can be isometrically deformed
 to $(T\circ f,-dT\circ \nu)$.
\end{remark}

\subsection{Realizations of Kossowski metrics with prescribed curvature lines}

We now construct a wave front whose first fundamental 
form is a given germ of Kossowski metric, and
with a given curve that is  a curvature line with a
prescribed normal curvature function.
For this purpose, we prepare the following fact, 
which is discussed in \cite{Kossowski}, 
\cite{MU}, and \cite{T}.
(Teramoto \cite{T} investigated 
the behavior of the principal 
curvature functions
near a non-degenerate singular point $p$
in terms of several geometric invariants
at $p$.)

\begin{fact}\label{thm:T}
 Let $f:U\to \R^3$ be a $C^r$-wave front, and $p\in U$ 
 a non-degenerate singular point whose limiting normal
 curvature does not vanish.
 Then there is a unique curvature line $\gamma$
 passing through $p$ such that the
 principal curvature function along it is bounded.
\end{fact}

We call $\gamma$ the \emph{characteristic principal curvature line}.

\begin{proof}
 Each non-degenerate singular point 
 is a regular point on a suitable
 parallel surface of a given
 wave front,
 and the principal curvature 
 lines are common in the parallel surfaces.
 It has been shown that
 umbilical points of regular surfaces
 cannot be a singular points of their 
 parallel surfaces, and 
 two distinct $C^r$-vector fields $X,Y$ 
 of principal directions are
 defined on a sufficiently small
 neighborhood of
 non-degenerate singular points
 (cf.\ \cite{Kossowski} and  
 \cite[Proposition 1.10]{MU}).
 Since $X_p,Y_p$ are linearly independent,
 we may assume that $X_p$ is not a null vector,
 without loss of generality. 
 Let $\gamma(t)$ be the integral curve
 of $X$ such that $\gamma(0)=p$.
 Then we can take a {\rm K}-orthogonal coordinate neighborhood
 $(U; u,v)$ such that $\gamma'(0)=\partial_u$.
 Since $\gamma$ is an integral curve of
 $X$, the normal curvature function $\kappa_n$ along $\gamma$
 gives the principal curvature.
 Since $\kappa_n$
 is less than or equal
 to the curvature of $f\circ \gamma$ as a space curve,
 \eqref{eq:kn}
 yields that the function $\kappa_n$ along $\gamma$ 
 is bounded.

 On the other hand, 
 since the limiting normal curvature
 does not vanish, $\Omega(p)\ne 0$ holds, by
 Fact \ref{fact:front}. 
 So the Gaussian curvature is 
 unbounded at $p$.
 So the principal curvature line 
 passing through $p$
 as an integral curve of $Y$ 
 has unbounded principal curvature.
\end{proof}

\begin{definition}
\label{rem:geod}
 Let $p$ be a semi-definite point of a Kossowski metric $ds^2$
 and $\epsilon$ a positive number.
 A regular curve $\gamma(t)$ ($t\in [0,\epsilon]$)
 emanating from $p(=\gamma(0))$
 is called a \emph{special geodesic} if
 \begin{itemize}
  \item $\gamma'(0)$ is not a null vector,
  \item $ds^2$ is positive definite at
	$\gamma(t)$ $(0<t\le \epsilon)$,
  \item $\gamma((0,\epsilon])$  
	is the image of a geodesic with respect to $ds^2$.
 \end{itemize}
\end{definition}

We shall now prove the following:

\begin{theorem}\label{thml:main2}
 Let $ds^2$ be a real analytic Kossowski metric.
 Suppose that 
 $p\in M^2$ is a regular point or
 a non-parabolic semi-definite point
 of $ds^2$.
 We set
 \[
   m_p:=
  \begin{cases}
   \sqrt{K_p} & \mbox{if $p$ is a regular point with $K_p(>0)$}, \\
      \infty & \mbox{otherwise}, 
   \end{cases}
 \]
 where $K_p$ denotes the Gaussian curvature at $p$.
 Let
 $\gamma(t)$ 
 $(|t|<\epsilon)$ 
 be a regular curve
 on $M^2$ 
 such that $\gamma(0)=p$ and $\gamma'(0)$
 is not a null vector.
 Take a germ $\omega(t)$ $(|t|<\epsilon)$
 of a real analytic 
 function 
 satisfying $e^{\omega(t)}<m_p$ for $|t|<\epsilon$.
 Then there exists a real analytic 
 immersion $($resp. a wave front$)$ 
 $f:U\to \R^3$ defined on a neighborhood of $p$
 such that
 $\gamma$ is a curvature line
 and $e^\omega$ is 
 the principal curvature function
 along $\gamma$ $($i.e. if $p$ is a semi-definite point,
 $\gamma$ is a characteristic principal curvature line$)$.
 The congruence class of
 $f$ is uniquely determined.
 Moreover, if $\gamma$ is a special 
 geodesic, $\hat \gamma:=f\circ \gamma$ is a planar curve.
\end{theorem}

\begin{proof}
 To adjust $\gamma$ 
 to be a curvature line, we set
 \begin{equation}\label{eq:CC}
  C(u,0)(=c(u))=0.
 \end{equation}
 Then $B(u,0)=0$ by Corollary \ref{cor:B}.
 By \eqref{eq:fufv0} and \eqref{eq:nunv0},
 the $u$-axis
 (i.e.\ $\gamma$) is a curvature line of $f$.
 If we set
 \[
   A(u,0)(=a(u)):=-e^{\omega(u)} \qquad (e^{\omega(u)}<m_p),
 \]
 then $e^{\omega(u)}$ coincides with
 the normal curvature function along $\gamma$.
 If we replace 
 $(a(u),c(u))=(-e^{\omega},0)$
 by $(e^{\omega},0)$, the congruence class
 of the  resulting wave front $f$ does not
 change (cf. Lemma \ref{rmk:U}).
 
 We next suppose that $\gamma$ is a special
 geodesic. Since $\alpha=0$, \eqref{eq:CC} and
 \eqref{eq:C} yield 
 $\mu(u)=0$,
 that is, $\hat \gamma$ lies in a plane.
\end{proof}

It should be remarked that special 
geodesics may not exist in general:

\begin{fact}[Remizov \cite{R}]\label{f:R}
 Let $p$ be a cuspidal edge on a wave front.
 If the singular curvature at $p$ is positive 
 {\rm(}resp.\ negative{\rm)},
 there are no {\rm(}resp.\ exactly two{\rm)} special 
 geodesics passing through $p$.
\end{fact}

Remizov investigated the geodesics 
of frontals whose singular set image consists of
regular space curves. 
Fact \ref{f:R} is a special case of
his result  \cite[Theorem~3]{R},
although he did not
formulate his results in terms of singular curvature.
We do not know if the above two special geodesics 
of cuspidal edges are real analytic or not 
when the wave front is real analytic.
Since a swallowtail can be considered as a limit of
cuspidal edges with negative singular curvature, it
can be expected that those two special geodesics converge to a
geodesic, and the following problem naturally arises:

\begin{question}\label{q:1}
 Is there a special geodesic at a given swallowtail?
\end{question}

Recently, Fukui \cite[Theorem 2.3 and
Remark 2.11]{F} showed the existence of 
a local coordinate system centered at each
swallowtail in $\R^3$ which shows that one of its 
coordinate line has the same $p$-th order
Taylor expansion
as the special geodesic, for each positive integer $p$.
In particular, the possibility of the existence of
a special geodesic is given
as a formal power series.

\section{Remaining problems}
\label{sec6}
In Corollary \ref{cor:germ-A2},
isometric deformations of cuspidal
edges and cuspidal cross caps 
that control their singular set images in $\R^3$
were obtained.
However, we cannot similarly discuss the
same problem for swallowtails, 
since the initial velocity  of the characteristic curve
is a null vector.
So the following question remains:

\begin{question}\label{q:2}
 For a given real analytic 
 space cusp $\hat \sigma$, is there a swallowtail
 having $\hat \sigma$ as the image of
 its singular set whose first fundamental form coincides with
 a given germ of non-parabolic $A_3$ semi-definite point of 
 a Kossowski metric?
\end{question}

Since a swallowtail is a limit point of
cuspidal edges, the last assertion 
of Corollary \ref{cor:germ-A2}
yields that the possibilities of such 
swallowtails are at most two.
However, the authors do not know of the existence of
two non-congruent swallowtails which have common 
first fundamental form and the same image for the singular set. 

\smallskip
By the way, 
the existence of isometric deformations of
cross caps is also an important remaining problem. 
In \cite{HHNUY}, non-trivial 
examples of isometric deformations of
cross caps are given. In \cite{HHNSUY},
a class of positive semi-definite metrics called
``Whitney metrics''
is defined. The first fundamental forms of cross caps are
Whitney metrics. So it is natural to ask:

\begin{question}\label{q:3}
 For a given real analytic germ of Whitney metric,
 is there a cross cap germ that is an isometric realization of it? 
\end{question}

In \cite{HNUY}, the authors found a solution 
of this as a formal power series, but
have not show the convergence.

\begin{ack}
 The authors thank the referee
 for valuable comments.
\end{ack}


\begin{thebibliography}{10}
 \bibitem{F}
  T.~Fukui,
  {\itshape Local differential geometry of cuspidal edge and swallowtail},
  to appear in Osaka J. Math.
  (www.rimath.saitama-u.ac.jp/lab.jp/Fukui/preprint/CE\_ST.pdf). 
 \bibitem{HHNSUY}
  M.~Hasegawa, A.~Honda, K.~Naokawa, K.~Saji, M.~Umehara and K.~Yamada,
  {\itshape Intrinsic properties of surfaces with singularities},
  Internat.\ J.\ Math.\ {\bfseries 26} (2015), 1540008.  
 \bibitem{HHNUY}
  M.~Hasegawa, A.~Honda, K.~Naowaka, M.~Umehara and K.~Yamada,
  {\itshape Intrinsic invariants of Cross Caps},
  Selecta Math.\ New Ser., {\bfseries 20} (2014), 769--785.
 \bibitem{HNUY}
  A.~Honda, K.~Naokawa, M.~Umehara and K.~Yamada,
  {\itshape Isometric realization of cross caps as formal power 
              series and its applications}, 
  Hokkaido Math.\ J., {\bfseries 48} (2019), 1--44.
 \bibitem{Brazil}
  A.~Honda, K. Naokawa, K. Saji, M. Umehara, and K. Yamada,
  {\itshape Duality on generalized cuspidal edges preserving 
              singular set images and first fundamental forms}, 
  preprint (arXiv:1906.02556). 
 \bibitem{HS}
  A.~Honda and K.~Saji,
  {\itshape Geometric invariants of $5/2$-cuspidal edges},
  Kodai Math.\ J., {\bfseries 42} (2019), 496--525.
 \bibitem{KRSUY}
  M.~Kokubu, W.~Rossman, K.~Saji, M.~Umehara and K.~Yamada, 
  {\itshape Singularities of flat fronts in hyperbolic space}, 
  Pacific J.\ Math., {\bfseries 221} (2005), 303--351;
  Addendum: {\itshape Singularities of flat fronts in hyperbolic space},
  Pacific J.\ Math., {\bfseries 294} (2018), 505--509.
 \bibitem{Kossowski}
  M.~Kossowski,
  {\itshape Realizing a singular first fundamental form as a 
	 nonimmersed surface in Euclidean 3-space},
  J.\ Geom., {\bfseries 81} (2004), 101--113.
 \bibitem{KP}
  S.~G. Krantz and H.~R. Parks,  
  {\sc A Primer of Real Analytic Functions},
  (Second Edition), Birkh\"auser 2002.
 \bibitem{MS}
  L.~Martins and K.~Saji, 
  {\itshape Geometric invariants of cuspidal edges},
  Can.\ J. Math., {\bfseries 68} (2016), 455--462.
 \bibitem{MSUY}
  L.~Martins, K.~Saji, M.~Umehara and K.~Yamada,
  {\itshape Behavior of Gaussian curvature and mean curvature
            near non-degenerate singular points on wave fronts},
	 Geometry and Topology of Manifolds, 
  pp.\ 247--281, 
  Springer Proc.\ Math.\ Stat., {\bfseries 154}, Springer, 2016.
 \bibitem{MU}
  S.~Murata and M.~Umehara, 
  {\itshape Flat surfaces with singularities in Euclidean 3-space}, 
  J. Diff.\ Geom., {\bfseries 221} (2005), 303--351.
 \bibitem{NUY}
  K.~Naokawa, M.~Umehara and K.~Yamada,
  {\itshape Isometric deformations of cuspidal edges},
  Tohoku Math.\ J., {\bfseries 68} (2016), 73--90.
 \bibitem{R}
  A.~O.~Remizov,
  {\itshape Singularities of a geodesic flow
	 on surfaces with a cuspidal edge},
   Proceedings of the Steklov Institute of Mathematics, 
  {\bfseries 268} (2010), 248--257.
 \bibitem{SUY_annals}
  K.~Saji, M.~Umehara and K.~Yamada,
  {\itshape The geometry of fronts},
  Ann.\ of Math., {\bfseries 169} (2009), 491--529.
 \bibitem{SUY_kodai}
  K.~Saji, M.~Umehara and K.~Yamada,
  {\itshape $A_2$-singularities of 
	 hypersurfaces with non-negative sectional curvature 
	 in Euclidean space},
  Kodai Math.\ J., {\bfseries 34} (2011), 390--409.
 \bibitem{SUY_msj}
  K.~Saji, M.~Umehara and K.~Yamada,
  {\itshape An index formula for a bundle homomorphism 
	 of the tangent bundle into a vector bundle of 
	 the same rank, and its applications},
  J. Math.\ Soc.\ Japan., {\bfseries 69}  
  (2017), 417-457.
 \bibitem{T}
  K.~Teramoto, 
  {\itshape Principal curvatures and parallel surfaces 
	 of wave fronts}, 
  Adv.\ in Geom., {\bf 19} (2019), 541--554.
 \bibitem{UY-book}
  M.~Umehara and K.~Yamada,
  {\sc Differential Geometry of Curves and Surfaces},
  2017, World Scientific Inc.
\end{thebibliography}
\end{document}